\documentclass[12pt,twoside]{article}
\usepackage{amsmath,amsfonts,amssymb,amsthm,array,mathdots}
\usepackage{a4,enumitem,indentfirst}
\usepackage{etoolbox} 
\usepackage{url} 

\usepackage{stmaryrd} 

\usepackage{epsfig, color} 

\graphicspath{{figures/}}

\theoremstyle{plain}

\newtheorem{theo}{Theorem}
\newtheorem{prop}{Proposition}[section]

\newtheorem{coro}[prop]{Corollary}
\newtheorem{fact}[prop]{Fact}
\newtheorem{lemma}[prop]{Lemma}

\theoremstyle{definition}
\newtheorem{example}[prop]{Example}
\AtEndEnvironment{example}{\null\hfill$\diamond$}%

\newtheorem{remark}[prop]{Remark}
\AtEndEnvironment{remark}{\null\hfill$\diamond$}%

\newcommand\ray{\operatorname{ray}}
\newcommand\per{\operatorname{per}}

\newcommand\card{\operatorname{card}}

\newcommand\Frenet{\mathfrak{F}}

\newcommand{\Word}{{\mathbf W}}
\newcommand{\Ideal}{{\mathbf I}}

\newcommand{\bfx}{{\mathbf x}}
\newcommand{\bfy}{{\mathbf y}}

\newcommand{\sing}{\operatorname{sing}}
\newcommand{\PA}{\operatorname{PA}}

\newcommand\SO{\operatorname{SO}}

\newcommand\GL{\operatorname{GL}}
\newcommand\Lo{\operatorname{Lo}}
\newcommand\Up{\operatorname{Up}}

\newcommand\so{\operatorname{\mathfrak{so}}}

\newcommand\Spin{\operatorname{Spin}}
\newcommand\spin{\mathfrak{spin}}

\newcommand\inv{\operatorname{inv}}
\newcommand\Inv{\operatorname{Inv}}
\newcommand\Diag{\operatorname{Diag}}
\newcommand\diag{\operatorname{diag}}
\newcommand\B{\operatorname{B}}
\newcommand\Quat{\operatorname{Quat}}

\newcommand{\resultant}{\operatorname{resultant}}
\newcommand{\discriminant}{\operatorname{discriminant}}

\newcommand{\longhat}{\operatorname{hat}}
\newcommand{\chop}{\operatorname{chop}}
\newcommand{\adv}{\operatorname{adv}}
\newcommand{\iti}{\operatorname{iti}}

\newcommand{\mult}{\operatorname{mult}}

\newcommand{\NN}{{\mathbb{N}}}
\newcommand{\ZZ}{{\mathbb{Z}}}

\newcommand{\RR}{{\mathbb{R}}}

\newcommand{\Ss}{{\mathbb{S}}}

\newcommand{\DD}{{\mathbb{D}}}

\newcommand{\PP}{{\mathbb{P}}}
\newcommand{\HH}{{\mathbb{H}}}

\newcommand{\cH}{{\cal H}}
\newcommand{\cL}{{\cal L}}

\newcommand{\cU}{{\cal U}}

\newcommand{\cD}{{\cal D}}
\newcommand{\cP}{{\cal P}}

\newcommand{\cM}{{\cal M}}
\newcommand{\cY}{{\cal Y}}

\newcommand{\bi}{{\mathbf{i}}}
\newcommand{\bj}{{\mathbf{j}}}
\newcommand{\bk}{{\mathbf{k}}}
\newcommand{\bL}{{\mathbf{L}}}
\newcommand{\bQ}{{\mathbf{Q}}}

\newcommand{\fa}{{\mathfrak a}}
\newcommand{\fb}{{\mathfrak b}}
\newcommand{\fc}{{\mathfrak c}}

\newcommand{\fl}{{\mathfrak l}}

\newcommand{\Pos}{\operatorname{Pos}}
\newcommand{\Neg}{\operatorname{Neg}}
\newcommand{\Bru}{\operatorname{Bru}}
\newcommand{\BL}{\operatorname{BL}}

\newcommand{\sign}{\operatorname{sign}}

\begin{document}

\title{Homotopy type of spaces of \\
locally convex curves in the sphere $\Ss^3$}
\author{Em{\'\i}lia Alves 
\and Victor Goulart
\and Nicolau C. Saldanha}
\date{}

\maketitle

\begin{abstract}
Locally convex (or nondegenerate) curves in the sphere $\Ss^n$
(or the projective space $\PP^n$)
have been studied for several reasons,
including the study of linear ordinary differential equations
of order $n+1$.
Taking Frenet frames allows us to obtain
corresponding curves $\Gamma$ in the group $\Spin_{n+1}$,
the universal cover of the space of flags.
Let $\cL_n(z_0;z_1)$ be the space 
of such curves $\Gamma$ with
prescribed endpoints $\Gamma(0) = z_0$, $\Gamma(1) = z_1$.
The aim of this paper is to
determine the homotopy type of the spaces 
$\cL_3(z_0;z_1)$ for all $z_0, z_1 \in \Spin_4
= \Ss^3 \times \Ss^3 \subset \HH \times \HH$;
$\HH$ is the ring of quaternions.
As a corollary, we obtain the homotopy type
of the space of closed locally convex curves in either
$\Ss^3$ or $\PP^3$.

There are many previous papers addressing related questions,
some of them by the authors.
An early such paper solves the corresponding problem for $n = 2$
(i.e., for curves in $\Ss^2$).
Another previous result (with B. Shapiro) reduces the problem
to $z_0 = 1$ and $z_1 \in \Quat_4$ where $\Quat_4 \subset \Spin_4$
is a finite group of order $16$
with center $Z(\Quat_4) = \{(\pm 1,\pm 1)\}$.
A more recent paper shows that for
$z_1 \in \Quat_4 \smallsetminus Z(\Quat_4)$
we have a homotopy equivalence
$\cL_3(1;z_1) \approx \Omega\Spin_4$.
In this paper we compute the homotopy type of $\cL_3(1;z_1)$
for $z_1 \in Z(\Quat_4)$:
it is equivalent to the wedge of
$\Omega\Spin_4$ with an infinite countable family of spheres
(as for the case $n = 2$).

The structure of the proof can be compared to that of the case $n = 2$
but some of the steps require the creation of new theories,
involving algebra and combinatorics.
We construct explicit subsets $\cY \subset \cL_n(z_0;z_1)$
for which the inclusion $\cY \subset \Omega\Spin_{n+1}(z_0;z_1)$
is a homotopy equivalence.
For $n = 2$, there is a simple geometric description of $\cY$;
for $n = 3$, the far less natural construction is based on the theory of
\textit{itineraries} of such curves.
The itinerary of a curve in $\cL_n(1;z_1)$
is a finite word in the alphabet $S_{n+1} \smallsetminus \{e\}$
of nontrivial permutations.
\end{abstract}

\medskip

\section{Introduction}
\label{section:intro}

Let $J \subset \RR$ be an interval.
A sufficiently smooth curve $\gamma: J \to \Ss^n \subset \RR^{n+1}$
is \emph{(positive) locally convex} 
\cite{Alves-Saldanha, Saldanha3, Saldanha-Shapiro}
or \emph{(positive) nondegenerate} 
\cite{Khesin-Ovsienko, Khesin-Shapiro2, Little} 
if it satisfies 
\[ \forall t \in J, \;
\det(\gamma(t),\gamma'(t),\ldots,\gamma^{(n)}(t))>0. \]
We are interested in the space of closed locally convex curves
and in the related spaces of curves with prescribed initial and final jet.

A smooth curve $\gamma: [0,1] \to \Ss^2$ is locally convex
if and only if it is an immersion
with positive geodesic curvature at every point
$t \in [0,1]$.
The main result in \cite{Saldanha3} can be stated as Fact~\ref{fact:mainS2} below.

Let $\cL_{2}(\per)$ be the space of closed locally convex curves
$\gamma: [0,1] \to \Ss^2$,
or, equivalently,
the space of $1$-periodic locally convex curves
$\gamma: \RR \to \Ss^2$.
Let $\cL_{2}(I;I)$ be the space of
locally convex curves $\gamma: [0,1] \to \Ss^2$ satisfying
\[ \gamma(0) = \gamma(1) = e_1, \qquad
\gamma'(0) = \gamma'(1) = e_2. \]

\begin{fact}
\label{fact:mainS2}
The space $\cL_{2}(\per)$ is homeomorphic to $\cL_{2}(I;I) \times \SO_3$.
The space $\cL_{2}(I;I)$ has three connected components, one contractible;
the other two have the homotopy type of the following spaces:
\begin{gather*}
\Omega\Ss^3 \vee \Ss^2 \vee \Ss^6 \vee \Ss^{10} \vee \cdots, \\
\Omega\Ss^3 \vee \Ss^4 \vee \Ss^8 \vee \Ss^{12} \vee \cdots.
\end{gather*}
\end{fact}

A smooth curve $\gamma: [0,1] \to \Ss^3$ is \emph{locally convex}
if it is an immersion with nonzero geodesic curvature
and positive torsion at every point $t \in [0,1]$.
The main result in the present paper is a corresponding result for $\Ss^3$,
requiring for its proof the development of several new tools.

Let $\cL_{3}(\per)$ be the space of closed locally convex curves
$\gamma: [0,1] \to \Ss^3$.
Let $\cL_{3}(I;I)$ be the space of
locally convex curves $\gamma: [0,1] \to \Ss^3$
satisfying
\[ \gamma(0) = \gamma(1) = e_1, \qquad
\gamma'(0) = \gamma'(1) = e_2, \qquad
\gamma''(0) = \gamma''(1) = e_3. \qquad
 \]
Let $\cL_{3}(I;-I)$ be the space of
locally convex curves $\gamma: [0,1] \to \Ss^3$
satisfying
\[ \gamma(0) = -\gamma(1) = e_1, \qquad
\gamma'(0) = -\gamma'(1) = e_2, \qquad
\gamma''(0) = -\gamma''(1) = e_3. \qquad
 \]

\begin{theo}
\label{theo:mainS3}
The space $\cL_{3}(\per)$ is homeomorphic to $\cL_{3}(I;I) \times \SO_4$.
The space $\cL_{3}(I;I)$ has two connected components
that have the homotopy type of the following spaces:
\begin{gather*}
\Omega(\Ss^3 \times \Ss^3) \vee
\Ss^2 \vee \Ss^6 \vee \Ss^6 \vee \Ss^{10} \vee \Ss^{10} \vee \Ss^{10}
\vee \cdots, \\
\Omega(\Ss^3 \times \Ss^3) \vee
\Ss^4 \vee \Ss^8 \vee \Ss^8 \vee \Ss^{12} \vee \Ss^{12} \vee \Ss^{12}
\vee \cdots.
\end{gather*}
The space $\cL_{3}(I;-I)$ has three connected components,
one contractible;
the other two have the homotopy type of the following spaces:
\begin{gather*}
\Omega(\Ss^3 \times \Ss^3) \vee
\Ss^2 \vee \Ss^6 \vee \Ss^6 \vee \Ss^{10} \vee \Ss^{10} \vee \Ss^{10}
\vee \cdots, \\
\Omega(\Ss^3 \times \Ss^3) \vee
\Ss^4 \vee \Ss^4 \vee \Ss^{8} \vee \Ss^{8} \vee \Ss^{8}
\vee \Ss^{12} \vee \Ss^{12} \vee \Ss^{12} \vee \Ss^{12} \vee \cdots.
\end{gather*}
\end{theo}

\begin{remark}
\label{remark:oddcurve}
For odd $n$, a curve $\gamma: [0,1] \to \PP^n$
is (positive) locally convex if either lift 
$\tilde\gamma: [0,1] \to \Ss^n$ is (positive) locally convex.
The space of closed locally convex curves
$\gamma: [0,1] \to \PP^3$ can be naturally decomposed
into two disjoint subpaces:
those which are homotopic to a point and those which are not.
The first one is homeomorphic to $\cL_3(I;I) \times (\SO_4/\{\pm I\})$;
the second one is homeomorphic to $\cL_3(I;-I) \times (\SO_4/\{\pm I\})$.
\end{remark}


\begin{remark}
\label{remark:convex}
Let $J \subset \RR$ be a compact interval.
A smooth curve $\gamma: J \to \Ss^3 \subset \RR^4$
is \textit{convex} (or \textit{disconjugate} \cite{Shapiro-Shapiro})
if for every hyperplane $H \subset \RR^4$
with $0 \in H$
the curve $\gamma$ crosses $H$ in at most $3$ points.
Crossings are counted with multiplicity and
only count for $t$ in the interior of $J$.
If $\gamma$ is convex then $\gamma$
is either positive or negative locally convex
(depending on the sign of torsion).
Conversely, if $\gamma$ is positive locally convex
then its restrictions to sufficiently short compact intervals $J$
are convex (justifying the name).
A smooth curve $\gamma: J \to \PP^3$ is convex if there exists
a smooth convex curve $\tilde\gamma: J \to \Ss^3$
with $\gamma = \Pi \circ \tilde\gamma$,
where $\Pi: \Ss^3 \to \PP^3$ is the usual double cover.
There exist closed convex curves in $\PP^3$ but not in $\Ss^3$.
\end{remark}

The following fact gives an application
of Theorem~\ref{theo:mainS3} above
for linear ordinary differential equations
\cite{Khesin-Ovsienko, Saldanha-Tomei}.
The reader should compare this discussion with Sturm theory.

A function $h: \RR \to \RR$
is \textit{$1$-periodic} (resp. \textit{$1$-antiperiodic})
if $h(t+1) = h(t)$ (resp. $h(t+1) = -h(t)$) for all $t \in \RR$.
For a triple $(c_0,c_1,c_2)$ of $1$-periodic functions,
consider the linear ODE
\[ y^{(4)}(t) + c_2(t) y''(t) + c_1(t) y'(t) + c_0(t) y(t) = 0. \]
The triple $(c_0,c_1,c_2)$ is
\textit{solution-periodic} (resp. \textit{solution-antiperiodic})
if all solutions $y$ of the ODE above are
$1$-periodic (resp. $1$-antiperiodic).
Let $\cP_{+}$ (resp. $\cP_{-}$) be the set of
solution-periodic (resp. solution-antiperiodic) triples.
The spaces $\cL_{3}(I;\pm I)$ are as in Theorem~\ref{theo:mainS3}.

\begin{fact}
\label{fact:ode}
The space $\cP_{+}$ is homotopy equivalent to $\cL_{3}(I;I)$;
the space $\cP_{-}$ is homotopy equivalent to $\cL_{3}(I;-I)$.
\end{fact}

\bigbreak

Section~\ref{section:context} puts the new results,
such as Theorem~\ref{theo:mainS3},
into the context of previous work.
In the process we are introduced to notation
which allows us to give more precise statements,
such as Theorem~\ref{theo:L3},
and to isolate the main difficulties, 
as in Proposition~\ref{prop:M}.
Section \ref{section:review} contains a review 
of notation and results from previous work,
including work by the authors.
This includes a somewhat detailed description
of the stratification $\cL_3[w] \subset \cL_3$
and of the dual CW complex $\cD_3$;
a far more detailed description of both the stratification
and the CW complex is given in \cite{Goulart-Saldanha1,Goulart-Saldanha-cw}.
This stratification is a prerequisite
for the construction of the subsets
$\cM_{\mu_0,\mu_1} \subset \cL_3$;
a quick description of these objects is given in Section~\ref{section:context}.
In Section \ref{section:positivity} we state and prove
a few necessary results about intersections of Bruhat cells.
We also review consequences of positivity.
The lemmas and propositions are stated and proved for arbitrary dimension
because restricting ourselves to low dimension would sound artificial.
On the other hand,
Examples~\ref{example:z0z1z2}, \ref{example:z0z1z2w}
and \ref{example:z0z1z2x} focus on the cases
which are directly relevant in later sections.  

In Section \ref{section:closed} we study topological properties
of the stratification and prove that
each $\cM_{\mu_0,\mu_1} \subset \cL_3$ is a closed subset.
Sections \ref{section:bacb} and \ref{section:eta}
are dedicated to proving that 
$\cM_{\mu_0,\mu_1} \subset \cL_3$
is a collared topological submanifold.
The case $\mu_0 = 1$ and $\mu_1 = 0$ is addressed in Section \ref{section:bacb}
and the case $\mu_0 = \mu_1 = 1$ in Section \ref{section:eta}.
The general case then follows easily (also in Section \ref{section:eta}).
In Section \ref{section:contractibility} we prove that 
each set $\cM_{\mu_0,\mu_1}$ is contractible:
the proof is based on considering a related finite dimensional problem.
The careful reader will notice that these claims
are all part of the statement of
Proposition \ref{prop:M}.
The maps $h_{\mu_0,\mu_1}$, also mentioned in 
Proposition \ref{prop:M},
are constructed in Section \ref{section:maps};
in this section we also complete the proof of
Proposition \ref{prop:M}.
The proofs of Theorems \ref{theo:mainS3} and \ref{theo:L3}
are now short and presented in Section \ref{section:theo}.

\bigbreak

The authors would like to thank Boris Shapiro for helpful conversations,
particularly concerning the study of examples such as 
the one shown in Figure~\ref{fig:bacb-2}.
Support from CNPq, CAPES and Faperj (Brazil) 
is gratefully acknowledged.

\bigbreak

\section{Context and statements}
\label{section:context}

In this section we put the results of the present paper
in the context of related work.
We also give more precise (but also more technical)
versions of the main results. For instance,
Theorem~\ref{theo:L3} is a more precise version of Theorem~\ref{theo:mainS3}.

\smallskip

It is useful to give an alternate description
of locally convex curves as curves in one of the groups
$\SO_{n+1}$ or $\Spin_{n+1}$.
A locally convex curve
$\gamma: J \to \Ss^n$ can be associated with
$\Gamma = \Frenet_{\gamma}: J \to \SO_{n+1}$ 
where the column-vectors of $\Gamma(t)$
are the result of applying the 
Gram-Schmidt algorithm to the ordered basis 
$(\gamma(t),\gamma'(t),\ldots,\gamma^{(n)}(t))$ of $\RR^{n+1}$. 
The curve $\Gamma$ can then be lifted 
to the double cover $\Spin_{n+1}$.
The curve $\Gamma: J \to \Spin_{n+1}$
is also called locally convex. 

Recall that the double cover of $\SO_4$
is $\Spin_4 = \Ss^3 \times \Ss^3 \subset \HH \times \HH$,
where $\HH$ is the ring of quaternions.
Indeed, the universal cover $\Pi: \Ss^3 \times \Ss^3 \to \SO_4$ is
\begin{equation}
\label{equation:Pi}
\Pi((z_l,z_r))w = z_l w z_r^{-1}, \qquad w \in \HH;
\end{equation}
here we identify $\RR^4 = \HH =
\{ x_1 + x_2 \bi + x_3 \bj + x_4 \bk; \; x_1, x_2, x_3, x_4 \in \RR \}$.

The corresponding curve $\Gamma: J \to \Spin_4$
is obtained as above and
its logarithmic derivative is of the form 
\begin{equation}
\label{equation:logdiff}
(\Gamma(t))^{-1}\Gamma'(t)=
\sum_{1 \le j\le 3}\kappa_j(t)\fa_j \in \spin_4, 
\end{equation} 
for positive functions $\kappa_1,\kappa_2,\kappa_3:J\to(0,+\infty)$. 
Here $\fa_j$ can be interpreted as a matrix in $\so_4$
(as in \cite{Goulart-Saldanha0})
or as a pair of quaternions
(as in \cite{Alves, Alves-Saldanha}):
\begin{equation}
\label{equation:faj}
\begin{gathered}
\fa_1 = \fa = \frac12 (\bi,-\bi) = e_2 e_1^\top - e_1 e_2^\top, \quad
\fa_2 = \fb = \frac12 (\bk,-\bk) = e_3 e_2^\top - e_2 e_3^\top, \\
\fa_3 = \fc = \frac12 (\bi,\bi) = e_4 e_3^\top - e_3 e_4^\top.
\end{gathered}
\end{equation}
Following \cite{Goulart-Saldanha0}, we
define homomorphisms $\alpha_j: \RR \to \Spin_4$ by
$\alpha_j(\theta) = \exp(\theta\fa_j)$.
We also write
\begin{equation}
\label{equation:acutegravehat}
\acute a_j =
\alpha_j\left(\frac{\pi}{2}\right), \quad
\grave a_j = (\acute a_j)^{-1} =
\alpha_j\left(- \frac{\pi}{2}\right), \quad
\hat a_j = (\acute a_j)^2 =
\alpha_j\left({\pi}\right), 
\end{equation}
so that we have $\hat a = \hat a_1 = (\bi,-\bi)$,
$\hat b = \hat a_2 = (\bk,-\bk)$
and $\hat c = \hat a_3 = (\bi,\bi)$.
The group $\Quat_4$ has order $16$ and
is generated by $\hat a$, $\hat b$ and $\hat c$.
For $\Pi$ as in Equation~\eqref{equation:Pi},
we have
$\Pi(\hat a) = \diag(-1,-1,1,1)$,
$\Pi(\hat b) = \diag(1,-1,-1,1)$ and
$\Pi(\hat c) = \diag(1,1,-1,-1)$.
The center of $\Quat_4$ is
$Z(\Quat_4) = \Pi^{-1}[\{\pm I\}] = \{ (\pm 1,\pm 1) \} =
\{\pm 1, \pm \hat a\hat c\}$
with $1 = (+1,+1)$, $-1 = (-1,-1)$,
$\hat a\hat c = (-1,+1)$ and $-\hat a\hat c = (+1,-1)$.

Given $z_0,z_1\in\Spin_{4}$, 
let $\cL_3(z_0;z_1)$ denote the space of 
locally convex curves $\Gamma:[0,1]\to\Spin_{4}$ 
with endpoints 
$\Gamma(0)=z_0$ and $\Gamma(1)=z_1$. 
More precisely,
for sufficiently large $r\in\NN$
we denote by
$\cL_3^{[C^r]}(z_0;z_1) \subset C^r([0,1];\HH^2)$
the space of such curves
of differentiability class $C^r$.
As will be discussed in Remark~\ref{remark:sufficientlysmooth} below, 
the value of $r$ does not affect the homotopy type of the space,
so we drop the superscript and write merely $\cL_3(z_0;z_1)$.
Left multiplication gives a homeomorphism
between $\cL_3(z_0;z_1)$ and $\cL_3(1;z_0^{-1}z_1)$:
from now on we restrict our attention to spaces of the form $\cL_3(1;z)$,
$z \in \Spin_4$.

It is proved in \cite{Goulart-Saldanha0,Saldanha-Shapiro} that
for every $z \in \Spin_{4}$ there exists $q \in \Quat_{4}$
such that $\cL_n(1;z)$ and $\cL_n(1;q)$
are homotopy equivalent
(the map from $z$ to $q$ is explicitly described).
We aim to study the homotopy type of $\cL_3(1;q)$, 
$q \in \Quat_{4}$, or equivalently that of
the disconnected space
\[ \cL_3 = \bigsqcup_{q \in \Quat_{4}} \cL_3(1;q). \]
For $z \in \Spin_4 = \Ss^3 \times \Ss^3$,
let $\Omega\Spin_4(1;z) \approx \Omega(\Ss^3\times\Ss^3)$
be the loop space of continuous paths $\Gamma: [0,1] \to \Spin_4$
with $\Gamma(0) = 1$, $\Gamma(1) = z$
so that $\cL_3(1;z) \subset \Omega\Spin_4(1;z)$.
Corollary 1.1 in \cite{Goulart-Saldanha-cw}
gives us a large part of the answer;
we restate it here, restricted to the case $n = 3$.

\begin{fact} 
\label{fact:coro11}
If $q \in \Quat_{4} \smallsetminus Z(\Quat_{4})$ then
the inclusion $\cL_3(1;q) \subset \Omega\Spin_{4}(1;q)$
is a homotopy equivalence.
\end{fact}

The aim of this paper is to describe 
the homotopy type of $\cL_3(1;q)$
for the remaining cases $q \in Z(\Quat_4)$.
The following theorem is a slightly more informative
restatement of Theorem \ref{theo:mainS3}.

\begin{theo}
\label{theo:L3}
We have the following weak homotopy equivalences:
\begin{align*}
\cL_3(1;1) &\approx 
\Omega(\Ss^3 \times \Ss^3) \vee \Ss^4 \vee \Ss^8 \vee \Ss^8
\vee \Ss^{12} \vee \Ss^{12} \vee \Ss^{12} \vee \cdots, \\
\cL_3(1;-1) &\approx 
\Omega(\Ss^3 \times \Ss^3) \vee \Ss^2 \vee \Ss^6 \vee \Ss^6
\vee \Ss^{10} \vee \Ss^{10} \vee \Ss^{10} \vee \cdots, \\
\cL_3(1;-\hat a\hat c) &\approx 
\Omega(\Ss^3 \times \Ss^3) \vee \Ss^0 \vee \Ss^4 \vee \Ss^4
\vee \Ss^{8} \vee \Ss^{8} \vee \Ss^{8} \vee \cdots, \\
\cL_3(1;\hat a\hat c) &\approx 
\Omega(\Ss^3 \times \Ss^3) \vee \Ss^2 \vee \Ss^6 \vee \Ss^6
\vee \Ss^{10} \vee \Ss^{10} \vee \Ss^{10} \vee \cdots.
\end{align*}
The above bouquets include one copy of $\Ss^k$,
two copies of $\Ss^{(k+4)}$, \dots, $j+1$ copies of $\Ss^{(k+4j)}$, \dots,
and so on.
\end{theo}

Notice that $\cL_3(1;-\hat a\hat c)$ is disconnected,
with one connected component being contractible
(the subset of convex curves);
the other three spaces are connected.
In the notation of Theorem~\ref{theo:mainS3}
we have $\cL_{3}(I;I) = \cL_3(1;1) \sqcup \cL_3(1;-1)$
and $\cL_{3}(I;-I) = \cL_3(1;-\hat a\hat c) \sqcup \cL_3(1;\hat a\hat c)$.


\begin{remark}
\label{remark:sufficientlysmooth}
In the statements of Theorems~\ref{theo:mainS3} and \ref{theo:L3}
we were vague as to what the space of curves is.
For instance, do the curves have to be of class $C^4$,
or perhaps even more smoothness is required?
Actually, for any usual topology for which
the definition of locally convex curve 
in Equation \eqref{equation:logdiff} makes sense,
including $\Gamma \in H^r$ ($r \ge 1$)
and $\Gamma \in C^r$ ($r \ge 1$),
the homotopy type of $\cL_3$ is the same.
This is discussed at length in
\cite{Goulart-Saldanha1}, Subsection~2.2;
see also Lemma~1.11 in \cite{Saldanha-Zuhlke1}.

However, for the present paper, it is more convenient
to work in a Hilbert manifold of sufficiently smooth curves,
such as $\Gamma \in H^r$, $r \ge 5$.
Theorem 4 in \cite{Goulart-Saldanha1}, in particular,
requires curves to be sufficiently smooth as above:
the corresponding problem for lower smoothness
is the Shapiro-Shapiro Grassmannian conjecture
(which is still open,
but see \cite{Saldanha-Shapiro-Shapiro} and
\cite{Saldanha-Shapiro-Shapiro1}).
From now on we assume that curves are in such a space.
\end{remark}

\begin{remark}
\label{remark:AS}
In \cite{Alves-Saldanha} we show how to represent
a locally convex curve $\gamma: J \to \Ss^3$
by an associate pair of curves
$\gamma_{l}, \gamma_{r}: J \to \Ss^2$
(which are of course far easier to draw).
The main results in \cite{Alves-Saldanha} are corollaries 
of Theorems~\ref{theo:mainS3} and \ref{theo:L3} in the present paper.
For instance, Theorem~A in \cite{Alves-Saldanha} 
(in the notation of the present paper) 
states that
\[ \dim(H^2(\cL_3(1;\hat a\hat c);\RR)) \ge 3, \qquad
\dim(H^4(\cL_3(1;-\hat a\hat c);\RR)) \ge 4. \]
Recall that $\dim(H^2(\Omega(\Ss^3 \times \Ss^3))) = 2$ and
$\dim(H^4(\Omega(\Ss^3 \times \Ss^3))) = 3$.
It follows easily from Theorems~\ref{theo:mainS3} and \ref{theo:L3}
that
\[ \dim(H^2(\cL_3(1;\hat a\hat c);\RR)) = 3, \qquad
\dim(H^4(\cL_3(1;-\hat a\hat c);\RR)) = 5. \]
We shall not use the pair-of-curves representation in the present paper.
It would of course be very interesting to translate
our constructions into this representation;
such a translation process appears to be nontrivial.
\end{remark}

An important difference between the case $n = 2$
and higher dimensional cases, including $n = 3$,
is that in the case $n = 2$ the combinatorics
of the symmetric group $S_3$
is so trivial that it is handled by considering a short list of cases
(as in Figures 3, 4 and 5 in \cite{Saldanha3}).
Already for $S_4$ (corresponding to the case $n = 3$)
this approach becomes impractical.
For $n \ge 3$ we need therefore new algebraic and combinatorial theories.
In particular, we need to talk about \textit{itineraries} of curves
and about \textit{parity alternating} permutations.
We proceed to sketch the most important constructions and results.

In \cite{Goulart-Saldanha1} we present a stratification
of the space $\cL_3$ by itineraries.
Let $\Word_3$ be the set of finite words
in the alphabet $S_4 \smallsetminus \{e\}$,
where $S_4$ is the symmetric group
and $e \in S_4$ is the trivial permutation.
Each curve $\Gamma \in \cL_3$ has an itinerary $\iti(\Gamma) \in \Word_3$.
In order to describe the itinerary of $\Gamma$,
first consider the decomposition of $\Spin_4$ into 
(unsigned, disconnected) Bruhat cells.
More precisely,
in the notation of \cite{Goulart-Saldanha0}
we have the stratification
\[ \Spin_4 = \bigsqcup_{\rho \in S_4} \Bru_\rho. \]
For $z \in \Spin_4$, we have $z \in \Bru_\rho$
if and only if there exist upper triangular matrices $U_l, U_r$
such that
\[ \Pi(z) = U_l P_\rho U_r; \]
here $\Pi: \Spin_4 \to \SO_4$ is the usual covering map
and $P_\rho$ is the permutation matrix,
$e_j^\top P_\rho = e_{j^\rho}^\top$.
Thus, $\Bru_\eta$ is an open dense subset of $\Spin_4$
where $\eta = abacba$ is the top permutation;
$a = a_1 = (12)$, $b = a_2 = (23)$ and $c = a_3 = (34)$ are generators of $S_4$.
The \textit{singular set} of $\Gamma$ is 
\begin{equation}
\label{equation:singular}
\sing(\Gamma) = \{ t \in (0,1) \;|\; \Gamma(t) \notin \Bru_\eta \};
\end{equation}
this set is always finite and depends continuously on $\Gamma$
(with the Hausdorff metric; this is Theorem 1 in \cite{Goulart-Saldanha1}).
Given $\Gamma \in \cL_3$, let $\ell = \card(\sing(\Gamma))$
and $\sing(\Gamma) = \{ t_1 < \cdots < t_\ell \} \subset (0,1)$:
the itinerary of $\Gamma$ is the word
$(\sigma_1,\ldots,\sigma_\ell) \in \Word_3$
where $\Gamma(t_j) \in \Bru_{\rho_j}$, $\rho_j = \eta\sigma_j$.
Itineraries allow us to define important open and closed
subsets of $\cL_3$.
Indeed, let
\begin{equation}
\label{equation:MY}
M = \{ aba, bacb, bcb, cba, abacba = \eta \}, \qquad
\tilde Y = S_4 \smallsetminus ( \{e\} \cup M);
\end{equation}
the tilde is there for notational consistency with \cite{Goulart-Saldanha-cw}.
The motivation for the choice of the set $M$ above is discussed
in Remarks~\ref{remark:whatisM} and \ref{remark:abc} 
(and in \cite{Goulart-Saldanha-cw}):
the set $M$ above is obtained from the subgroup
\begin{equation}
\label{equation:SPA}
S_{\PA} = \{ e, aba, bacb, bcb, ac, abc, cba, \eta = abacba \} < S_4 
\end{equation}
of parity alternating permutations by removing a few elements;
see Equation~\eqref{equation:parityalternating}
for the general definition of $S_{\PA}$.

\goodbreak

Let $M^\ast \subset \Word_3$ be the set of words
using only letters from $M$
(this includes the empty word).
Let $\cM \subset \cL_3$ be
the closed subset of words with itineraries in $M^\ast$.
A large part of the present paper is devoted 
to understanding the subset $\cM \subset \cL_3$.
Its complement $\tilde\cY_3 = \cL_3 \smallsetminus \cM \subset \cL_3$
is the open set of curves 
whose itinerary admits at least one letter in $\tilde Y$.
Notice that if $q \in \Quat_4 \smallsetminus Z(\Quat_4)$
then $\cL_3(1;q) \subset \tilde\cY_3$.
We recall one of the main results from
\cite{Goulart-Saldanha-cw} (Theorem~5):

\begin{fact} 
\label{fact:tildeY}
For each $q \in \Quat_4$, the inclusion 
$(\tilde\cY_3 \cap \cL_3(1;q)) \subset \Omega\Spin_4(1;q)$
is a weak homotopy equivalence.
\end{fact}

Define $\mu: M \to \NN^2$ 
(where $\NN = \{0,1,2,\ldots\}$)
by
\[ \mu(aba) = \mu(bacb) = \mu(bcb) = (1,0), \quad
\mu(cba) = (0,1), \quad \mu(abacba) = (1,1). \]
For $\Gamma \in \cM$ with $\iti(\Gamma) = \sigma_1\cdots\sigma_\ell$,
define $\mu(\Gamma) = \mu(\sigma_1) + \cdots + \mu(\sigma_\ell)$.
For $(\mu_0,\mu_1) \in \NN^2$,
let $\cM_{\mu_0,\mu_1} \subset \cM$ 
be the set of curves $\Gamma \in \cM$ with 
$\mu(\Gamma) = (\mu_0,\mu_1)$.
Notice that $\cM_{0,0} \subset \cL_3(1;-\hat a\hat c)$
is the contractible connected component of convex curves.
The following is one of the main results of the present paper.

\begin{prop}
\label{prop:M}
Consider $(\mu_0,\mu_1) \in \NN^2$ and
set $z_1 = (-1)^{(\mu_0+1)}(\hat a\hat c)^{(\mu_1+1)} \in Z(\Quat_4)$.
The subset $\cM_{\mu_0,\mu_1} \subset \cL_3$ is non empty, closed
and is a contractible topological submanifold
of codimension $2\mu_0 + 2\mu_1$ of $\cL_3(1;z_1)$.
There exists a continuous map
$h_{\mu_0,\mu_1}: \Ss^{2\mu_0+2\mu_1} \to \cL_3$
with the following properties:
\begin{enumerate}
\item{the map $h_{\mu_0,\mu_1}$ is homotopic to a constant
in $\Omega\Spin_4(1;z_1)$;}
\item{the image of $h_{\mu_0,\mu_1}$ is contained
in $\tilde\cY_3 \cup \cM_{\mu_0,\mu_1}$;}
\item{the map $h_{\mu_0,\mu_1}$ intersects $\cM_{\mu_0,\mu_1}$
precisely once, and that intersection is topologically transversal.}
\end{enumerate}
\end{prop}

Here and elsewhere, the topological submanifolds we meet are rather tame:
they may have a few edges and corners, but nothing more complicated than that.
The reader should compare this proposition
with Lemma~7.3 in \cite{Saldanha3}.
In particular, the functions $h_{\mu_0,\mu_1}$
above correspond to the functions $h_\ast$ in that case.

\bigbreak

\section{Review of previous results}
\label{section:review}

The general structure of the proof of our Theorem~\ref{theo:mainS3} follows
roughly the same line as the proof of Theorem~1 in \cite{Saldanha3}
(restated above as Fact~\ref{fact:mainS2}).
The hardest result there is the case $n=2$ of Fact~\ref{fact:tildeY};
the case $n=3$ (i.e., for curves in $\Ss^3$, as in the present paper)
is proved in \cite{Goulart-Saldanha-cw}.
Other parts of the argument, which are easy in the case $n=2$,
are much harder in the case $n=3$, and take up most of the present paper.
The cases $n\ge 4$ (which we do not discuss here) appear to be even harder.
In this section we review necessary facts and notation,
starting with more classical material
and following with our own results.

The symmetric group $S_{n+1}$ is always taken
with generators $a_i = (i,i+1)$ (a~transposition),
$1 \le i \le n$.
For $n = 3$, we write $a = a_1$, $b = a_2$ and $c = a_3$.
A \textit{reduced word} for a permutation
$\sigma \in S_{n+1}$ is a word of minimal length.
This minimal length equals $\inv(\sigma)$,
the number of inversions of the permutation $\sigma$.
More precisely, the set of inversions of $\sigma \in S_{n+1}$ is
\begin{equation}
\label{equation:inv}
\Inv(\sigma) = \{ (j_0,j_1) \;|\; 1 \le j_0 < j_1 \le n+1,
j_0^\sigma > j_1^\sigma \}
\end{equation}
and $\inv(\sigma) = \card(\Inv(\sigma))$.
If $\inv(\sigma_1) = \ell_1 = 1+\inv(\sigma_0)$,
we write $\sigma_0 \vartriangleleft \sigma_1$
if there exist reduced words
\begin{equation}
\label{equation:triangle}
\sigma_0 = a_{i_1} \cdots a_{i_{k-1}} a_{i_{k+1}} \cdots a_{i_{\ell_1}},
\qquad
\sigma_1 = a_{i_1} \cdots a_{i_{k-1}} a_{i_k} a_{i_{k+1}}
\cdots a_{i_{\ell_1}}. 
\end{equation}
The \textit{strong Bruhat order} 
\cite{Bjorner-Brenti} is the transitive closure
of $\vartriangleleft$: 
if $k \ge 0$ and $\inv(\sigma_k) = k+\inv(\sigma_0)$
we have $\sigma_0 \le \sigma_k$ if and only if there exist
$\sigma_1, \ldots, \sigma_{k-1}$ with
\[ \sigma_0 \vartriangleleft \sigma_1 \vartriangleleft \cdots
\vartriangleleft \sigma_{k-1} \vartriangleleft \sigma_k. \]
The top permutation under this order is $\eta \in S_{n+1}$,
$k^\eta = n+2-k$, $\inv(\eta) = n(n+1)/2$.
For $n = 3$, we have $\eta = abacba = [4321] = (14)(23)$
and $\inv(\eta) = 6$.

Let $\Up_{n+1} \subset \GL_{n+1}$ be the group
of invertible upper triangular matrices.
Two matrices $A_0, A_1 \in \GL_{n+1}$ are \textit{Bruhat equivalent}
if and only if there exist $U_0, U_1 \in \Up_{n+1}$
for which $A_1 = U_0 A_0 U_1$.
The equivalence classes are called \textit{Bruhat} or \textit{Schubert cells}.
Permutation matrices $P_\sigma$,
$e_j^\top P_\sigma = e_{j^\sigma}^\top$,
give representatives for such classes:
for $\sigma \in S_{n+1}$, the equivalent class of $P_\sigma$
is $\Bru_\sigma \subset \GL_{n+1}$.
We have
\[ \Bru_{\sigma_0} \subseteq \overline{\Bru_{\sigma_k}} \quad\iff\quad
\sigma_0 \le \sigma_k \]
(in the strong Bruhat order)
\cite{Bjorner-Brenti, Chevalley, Verma}.
A minor modification \cite{Goulart-Saldanha0}
of this equivalence relation
is to consider either
$\GL_{n+1}^{+} \subset \GL_{n+1}$ (positive determinant)
or, more frequently for us, $\SO_{n+1} \subset \GL_{n+1}$
and the subgroup
$\Up_{n+1}^{+} \subset \Up_{n+1}$ 
of triangular matrices with positive entries along the diagonal.
The corresponding equivalence relation
in either group is defined as follows:
two matrices $Q_0, Q_1$
are (signed) Bruhat equivalent
if and only if there exist $U_0, U_1 \in \Up_{n+1}^{+}$
for which $Q_1 = U_0 Q_0 U_1$.
We also call the equivalent classes \textit{Bruhat cells};
the group $\B_{n+1}^{+}$ of signed permutation matrices 
with positive determinant provides us with a set of representatives.
Thus, for $Q \in \B_{n+1}^{+}$,
$\Bru_Q \subset \SO_{n+1}$ is the equivalence class of $Q$.
We have $\Bru_Q \subset \Bru_\sigma$
where the permutation $\sigma \in S_{n+1}$ satisfies
$e_i^\top Q = \pm e_{i^\sigma}^\top$,
(thus defining a surjective homomorphism $\Pi: \B_{n+1}^{+} \to S_{n+1}$,
$\Pi(Q) = \sigma$).
More precisely, we have
\[ \Bru_\sigma =
\bigsqcup_{Q \in \Pi^{-1}[\{\sigma\}] \subset \B_{n+1}^{+}}
\Bru_Q, \]
for the same $\Pi$.
Such Bruhat cells $\Bru_Q$ are contractible
and the dimension of $\Bru_Q \subset \SO_{n+1}$ is $\inv(\sigma)$.

The Lie group we most frequently work with is $\Spin_{n+1}$,
the double cover of $\SO_{n+1}$ under (another)
surjective homomorphism $\Pi: \Spin_{n+1} \to \SO_{n+1}$.
The finite group
$\tilde \B_{n+1}^{+} = \Pi^{-1}[\B_{n+1}^{+}] \subset \Spin_{n+1}$
has generators $\acute a_i$, $1 \le i \le n$,
defined in Equation~\eqref{equation:acutegravehat} above
and Equation~(2) in \cite{Goulart-Saldanha0}.
The subgroup $\Quat_{n+1} < \tilde \B_{n+1}^{+}$
has generators $\hat a_i$, $1 \le i \le n$.
Given $\sigma = a_{i_1} \cdots a_{i_\ell}$
(a reduced word for a permutation $\sigma \in S_{n+1}$),
define
\[ \acute\sigma = \acute a_{i_1} \cdots \acute a_{i_\ell}, \quad
\grave\sigma = \grave a_{i_1} \cdots \grave a_{i_\ell}
\in \tilde \B_{n+1}^{+}: \]
these turn out to be independent of the choice of reduced word.
Define the Bruhat cell $\Bru_z \subset \Spin_{n+1}$, 
$z \in \tilde \B_{n+1}^{+}$,
as the connected component of $\Pi^{-1}[\Bru_Q]$ containing $z$,
where $Q = \Pi(z) \in \B_{n+1}^{+}$ and $\Bru_Q \subset \SO_{n+1}$.
We thus have the decomposition
\[ \Spin_{n+1} = \bigsqcup_{z \in \tilde\B_{n+1}^{+}} \Bru_z. \]
The subsets $\Bru_z \subset  \Spin_{n+1}$ ($z \in \tilde\B_{n+1}^{+}$)
are contractible submanifolds and
the open cells in the above decomposition are
$\Bru_{q\acute\eta} \subset \Spin_{n+1}$, $q \in \Quat_{n+1}$.

We have $\grave\eta = (\acute\eta)^{-1} \in \tilde\B_{n+1}^{+}$;
let $\cU_1 = \grave\eta \Bru_{\acute\eta} \subset \Spin_{n+1}$,
a contractible open neighborhood of $1 \in \Spin_{n+1}$.
More generally, for $z \in \Spin_{n+1}$,
let $\cU_z = z\,\cU_1 \subset \Spin_{n+1}$,
an open neighborhood of $z \in \Spin_{n+1}$;
if $z \in \tilde\B_{n+1}^{+}$ we have $\Bru_z \subseteq \cU_z$.
Let $\Lo_{n+1}^{1}$ be the nilpotent group of
lower triangular matrices with diagonal entries equal to $1$.
The usual $LU$ and $QR$ decompositions define homeomorphisms
$\bQ: \Lo_{n+1}^{1} \to \cU_1 \subset \Spin_{n+1}$ and
$\bL = \bQ^{-1}: \cU_1 \to \Lo_{n+1}^{1}$.
Intersections $\cU_1 \cap \Bru_z$,
or, more generally, $\cU_{z_0} \cap \Bru_{z_1}$,
$z, z_0, z_1 \in \tilde\B_{n+1}^{+}$,
have been extensively studied
\cite{Shapiro-Shapiro-Vainshtein1, Alves-Saldanha2};
we shall have more to say about this in Section \ref{section:positivity}.

Let $J$ be an open interval, $t_0 \in J$ and
$\Gamma: J \to \Spin_{n+1}$ a locally convex curve.
Assume that $\Gamma(t_0) \in \Bru_{z_0}$
where ${z_0} \in \tilde\B_{n+1}^{+}$ and
$\Bru_{z_0} \subset \Spin_{n+1}$
is a cell of positive codimension (i.e., not open).
Then there exist
$z_{-}, z_{+} \in \acute\eta\Quat_{n+1} \subset \tilde\B_{n+1}^{+}$
which are functions of $z_0$ only (not of $\Gamma$)
and $\epsilon > 0$ such that
$t \in (t_0 - \epsilon,t_0)$ implies $\Gamma(t) \in \Bru_{z_{-}}$
and
$t \in (t_0, t_0 + \epsilon)$ implies $\Gamma(t) \in \Bru_{z_{+}}$.
We write
\begin{equation}
\label{equation:advchop}
z_{-} = \chop(z_0), \quad z_{+} = \adv(z_0), \qquad
z_{-}, z_{+} \in \acute\eta\Quat_{n+1} \subset \tilde\B_{n+1}^{+}; 
\end{equation}
the claim above is Theorem~3 in \cite{Goulart-Saldanha0}.
A purely algebraic definition of the maps $\chop$ and $\adv$
is given in Equation~(3) in \cite{Goulart-Saldanha0};
it implies $\adv(q\acute\rho) = q\acute\eta$
(for all $q \in \Quat_{n+1}$ and $\rho \in S_{n+1}$).

Let $\Word_n$ be the set of finite words in $S_{n+1} \smallsetminus \{e\}$.
For $q \in \Quat_{n+1}$,
let $\cL_n(1;q)$ be the set of locally convex curves
$\Gamma: [0,1] \to \Spin_{n+1}$ with 
$\Gamma(0) = 1$ and $\Gamma(1) = q$.
Let 
\[ \cL_n = \bigsqcup_{q \in \Quat_{n+1}} \cL_n(1;q). \]
Recall from Equation~\eqref{equation:singular}
that $\sing(\Gamma) \subset (0,1)$ is the set of
$t \in (0,1)$ for which $\Gamma(t) \notin \Bru_\eta$.
The set $\sing(\Gamma)$ is the set of
\textit{moments of non transversality}, i.e.,
the set of moments $t$ when the osculating flag
(defined by the Frenet frame $\Gamma(t)$)
is non transversal to the initial flag
(which is defined by the canonical base).
If $\sing(\Gamma) = \{t_1 < \cdots < t_\ell \}$,
the \textit{itinerary} of $\Gamma$ is the word
$\iti(\Gamma) = (\sigma_1,\ldots,\sigma_\ell) \in \Word_n$
where $\Gamma(t_i) \in \Bru_{\eta\sigma_i}$.
The set of words with itinerary
$w = (\sigma_1, \ldots, \sigma_\ell) \in \Word_n$
is $\cL_n[w] \subset \cL_n$, a contractible Hilbert manifold
of codimension 
\begin{equation}
\label{equation:dimw}
\dim(w) = \sum_{1 \le k \le \ell} (\inv(\sigma_k) - 1) 
\end{equation}
(this sums up results of \cite{Goulart-Saldanha1}).
The stratification
\[ \cL_n = \bigsqcup_{w \in \Word_n} \cL_n[w] \]
admits a kind of dual, a CW complex $\cD_n$
with one cell of dimension $\dim(w)$ for each word $w \in \Word_n$.
In \cite{Goulart-Saldanha-cw} we construct an inclusion
$i: \cD_n \to \cL_n$ which is also a homotopy equivalence;
for $w \in \Word_n$, the cell $c_w$ in $\cD_n$
intersects $\cL_n[w]$ (under the map~$i$) exactly once
and in a transversal manner
(this sums up results of \cite{Goulart-Saldanha-cw}).

The glueing map for the CW complex $\cD_n$
is unfortunately not as well understood as would be desired.
We do, however, define in Section~4 of \cite{Goulart-Saldanha-cw}
partial orders $\preceq$ and $\sqsubseteq$
in $\Word_n$ with the property that the image of $c_w$
(under the map~$i$) is contained in the union of $\cL_n[\tilde w]$
for $\tilde w \preceq w$ and $\tilde w \sqsubseteq w$.
The partial order $\sqsubseteq$ is easy to compute
and $\preceq$ can also be completely computed for the case $n = 3$,
the case of interest in the present paper.
For the convenience of the reader,
we define here the partial order $\sqsubseteq$.
Let $w \in \Word_n$ and $\sigma \in S_{n+1}$ we define
\begin{equation}
\label{equation:sqsubseteq}
(w \sqsubseteq \sigma) \quad\iff\quad
((w \ne ()) \land (\mult(w) \le \mult(\sigma)) \land (\hat w = \hat\sigma)). 
\end{equation}
Here $()$ denotes the empty word,
\[ \mult(\sigma) = (1^\sigma - 1, 1^\sigma + 2^\sigma - (1 + 2), \ldots)
\in \NN^n, \qquad
\hat\sigma = \acute\sigma(\grave\sigma)^{-1} \in \Quat_{n+1}. \]
The partial order in $\NN^n$ is defined coordinate by coordinate:
$(x_1, \ldots, x_n) \le (y_1, \ldots, y_n)$ if and only if
$x_i \le y_i$ for all $i$.
For $w = (\sigma_1,\ldots,\sigma_\ell)$,
we have $\mult(w) = \sum_i \mult(\sigma_i)$
and $\hat w = \hat\sigma_1 \ldots \hat\sigma_\ell$.
For words $w_0$ and $w_1 = (\sigma_1, \ldots, \sigma_\ell) \in \Word_n$,
we have $w_0 \sqsubseteq w_1$ if and only if there exists non empty words
$w_{01}, \ldots, w_{0\ell}$ such that
$w_0$ equals the concatenation $w_{01}\cdots w_{0\ell}$
and $w_{0j} \sqsubseteq \sigma_j$ for every $j$, $1 \le j \le \ell$.


A permutation $\sigma \in S_{n+1}$ is \textit{parity alternating}
if and only if
\begin{equation}
\label{equation:parityalternating}
\forall k, \quad 1 \le k \le n \quad \to \quad
\left( k^\sigma \not\equiv (k+1)^\sigma \pmod 2 \right). 
\end{equation}
The set $S_{\PA} \subseteq S_{n+1}$ of parity alternating permutations
is a subgroup; for $n = 3$, $S_{\PA}$
is given in Equation~\eqref{equation:SPA}.
A permutation $\sigma \in S_{n+1}$ is parity alternating
if and only if $\hat\sigma \in Z(\Quat_{n+1})$
(Lemma~2.2 in \cite{Goulart-Saldanha-cw}).

Theorem~3 in \cite{Goulart-Saldanha-cw} makes clear
the importance of this concept.
We state it here as Fact~\ref{fact:Y}.
Let $Y = S_{n+1} \smallsetminus S_{\PA}$,
the set of permutations which are \textit{not} parity alternating.
Let $\Ideal_Y \subset \Word_n$ be the set of words
in $S_{n+1} \smallsetminus \{e\}$
containing at least one letter in $Y$.
For $q \in \Quat_{n+1}$, let $\cY_n(1;q) \subset \cL_n(1;q)$
be the set of curves in this space with itinerary in $\Ideal_Y$.

\begin{fact}
\label{fact:Y}
For $q \in \Quat_{n+1} \smallsetminus Z(\Quat_{n+1})$,
we have $\cY_n(1;q) = \cL_n(1;q)$.
For $q \in Z(\Quat_{n+1})$, $\cY_n(1;q) \subset \cL_{n}(1;q)$
is an open subset, dense in the component of non-convex curves.
In all cases, the inclusion $\cY_n(1;q) \subset \Omega\Spin_{n+1}(1;q)$
is a homotopy equivalence.
\end{fact}

Fact~\ref{fact:tildeY} in the introduction is a minor adaptation
of the case $n=3$ of Fact~\ref{fact:Y},
hence the set $\tilde Y$ is considered a variation of the set $Y$.

\begin{remark}
\label{remark:whatisM}
The set $M = \{ aba, bacb, bcb, cba, \eta \}$
(as in Equation \eqref{equation:MY}) is a large subset of $S_{\PA}$:
$M = S_{\PA} \smallsetminus \{ e, ac, abc \}$.
The element $e$ does not appear in itineraries.
The removal of the pair $\{ac,abc\}$ is an operation similar to a collapse:
such operations appear in the proofs of
Theorems~3 and 5 in \cite{Goulart-Saldanha-cw}.
See also Remark~\ref{remark:abc}.
\end{remark}

\bigbreak

\section{Cells and Positivity}
\label{section:positivity}

We return to the study of intersections 
$\cU_{z_0} \cap \Bru_{z_1}$ for
$z_0, z_1 \in \tilde\B_{n+1}^{+}$.
Recall that $\bQ$ and $\bL$ define diffeomorphisms
between $\cU_{1} \subset \Spin_{n+1}$ and $\Lo_{n+1}^{1}$.
The maps $\bQ$ and $\bL$ respect the Bruhat stratification:
if $\Pi(z) = U_0 P_\rho U_1$,
for $\Pi: \Spin_{n+1} \to \SO_{n+1}$ and $U_0, U_1 \in \Up_{n+1}$,
then $\bL(z) = U_0 P_\rho \tilde U_1$, $\tilde U_1 \in \Up_{n+1}$
(and vice versa).
These maps therefore define an identification between
subsets of $\cU_1$ and of $\Lo_{n+1}^{1}$.
In particular, a smooth curve $\Gamma: J \to \Lo_{n+1}^1$
is \textit{convex} if and only if
$\bQ \circ \Gamma: J \to \cU_{1} \subset \Spin_{n+1}$
is (locally) convex.
Equivalently, $\Gamma: J \to \Lo_{n+1}^1$
is convex if and only if
the logarithmic derivative $(\Gamma(t))^{-1} \Gamma'(t)$
can be written in the form
\begin{equation}
\label{equation:convexLo}
(\Gamma(t))^{-1} \Gamma'(t) =
\sum_{1 \le j \le n} \beta_j(t) \fl_j, 
\qquad \beta_j(t) > 0; 
\end{equation}
here $\fl_j$ is the matrix whose only nonzero entry
is $(\fl_j)_{j+1,j} = 1$
(\cite{Goulart-Saldanha0}, Section~4).

We are particularly interested in
$\Pos_\eta \subset \Lo_{n+1}^{1}$,
the semigroup of totally positive matrices
\cite{Ando, Berenstein-Fomin-Zelevinsky, Goulart-Saldanha0},
a contractible connected component of
$\bQ^{-1}[\,\cU_1 \cap \Bru_{\acute\eta}]$.
We give two equivalent characterizations.
As in \cite{Berenstein-Fomin-Zelevinsky, Goulart-Saldanha0},
$\lambda_j: \RR \to \Lo_{n+1}^1$ is a homomorphism:
$\lambda_j(t) = \exp(t\fl_j)$, so that
the only nonzero subdiagonal entry is $(\lambda_j(t))_{j+1,j} = t$.
For $m = n(n+1)/2 = \inv(\eta)$, let 
$\eta = a_{i_1} \cdots a_{i_m}$ 
be a reduced word for $\eta$.
We have
\begin{equation}
\label{equation:positiveeta}
\Pos_{1,\eta} = \Pos_\eta = \{ \lambda_{i_1}(t_1) \cdots \lambda_{i_m}(t_m); \;
t_1, \ldots, t_m \in (0,+\infty) \subset \RR \}. 
\end{equation}
Equivalently,
for $L \in \Lo_{n+1}^{1}$ we have $L \in \Pos_\eta$
if and only if
there exists a smooth convex curve
$\Gamma: [0,1] \to \Lo_{n+1}^{1}$ 
satisfying 
\begin{equation}
\label{equation:positiveconvex}
\Gamma(0) = I, \quad \Gamma(1) = L.
\end{equation}
The boundary of $\Pos_\eta \subset \Lo_{n+1}^{1}$ is
\begin{equation}
\label{equation:positive}
\partial{\Pos_\eta} =
\bigsqcup_{\rho \in S_{n+1} \smallsetminus \{\eta\}} \Pos_\rho,
\qquad
\Pos_\rho =
\overline\Pos_\eta \cap \bQ^{-1}[\,\cU_1 \cap \Bru_{\acute\rho}]. 
\end{equation}
Equivalently, if $\rho = a_{i_1} \cdots a_{i_\ell}$ is a reduced word
(so that $\ell = \inv(\rho)$) then
\begin{equation}
\label{equation:positiverho}
\Pos_{1,\rho} =
\Pos_\rho = \{ \lambda_{i_1}(t_1) \cdots \lambda_{i_\ell}(t_\ell); \;
t_1, \ldots, t_\ell \in (0,+\infty) \subset \RR \}. 
\end{equation}
If 
$\rho_0 = a_{i_1} \cdots a_{i_{k-1}} a_{i_{k+1}} \cdots a_{i_\ell}
\vartriangleleft
\rho_1 = a_{i_1} \cdots a_{i_{k-1}} a_{i_k} a_{i_{k+1}} \cdots a_{i_\ell}$
(where these are also reduced words)
then the map
\begin{equation}
\label{equation:boundarypos}
\begin{aligned}
(0,+\infty)^{k-1} \times [0,+\infty) \times (0,+\infty)^{\ell-k}
&\to \Pos_{\rho_0} \cup \Pos_{\rho_1} \subset \Lo_{n+1}^1, \\
(t_1,\ldots,t_l) &\mapsto \lambda_{i_1}(t_1) \cdots \lambda_{i_\ell}(t_\ell)
\end{aligned}
\end{equation}
is a homeomorphism and a smooth embedding in the interior.
Thus,
the union $\Pos_{\rho_0} \cup \Pos_{\rho_1} \subset \Lo_{n+1}^1$
is a topological manifold with boundary
$\Pos_{\rho_0}$.

For each $\rho \in S_{n+1}$,
set
$\Pos^{\Spin}_{1,\rho} = \bQ[\Pos_{1,\rho}] =
\bL^{-1}[\Pos_{1,\rho}] \subset \cU_1$.
Thus, $\Pos^{\Spin}_{1,\eta} = \bQ[\Pos_{1,\eta}] \subset \cU_1$
is an open subset and
a contractible connected component of $\cU_1 \cap \Bru_{\acute\eta}$.
The boundary of
$\Pos^{\Spin}_{1,\eta} \subset \cU_1$
is the disjoint union of the sets
$\Pos^{\Spin}_{1,\rho} = \bQ[\Pos_\rho] \subset \cU_1$,
$\rho \in S_n \smallsetminus \{\eta\}$.
(The boundary of
$\Pos^{\Spin}_{1,\eta} \subset \cU_1$
should not be confused with the larger boundary of
$\Pos^{\Spin}_{1,\eta} \subset \Spin_{n+1}$.)
Each subset $\Pos^{\Spin}_{1,\rho} \subset \cU_1$
is a submanifold of dimension $\inv(\rho)$ and
a contractible subset of $\cU_1 \cap \Bru_{\acute\rho}$.
We only have $\Pos^{\Spin}_{1,\rho} = \cU_1 \cap \Bru_{\acute\rho}$
in a few low-dimensional cases.

Consider $z_0 = q_0 \acute\rho_0 \in  \tilde\B_{n+1}^{+}$,
$\rho_0 \in S_{n+1}$ and $q_0 \in \Quat_{n+1}$.
Recall that $\Bru_{z_0} \subseteq \cU_{z_0}$.
For $\Pi: \Spin_{n+1} \to \SO_{n+1}$, set $Q_0 = \Pi(z_0) \in \SO_{n+1}$.
Let $Q_0 \Lo_{n+1}^1 \subset \GL_{n+1}$,
a left coset and an affine space of real matrices;
we also write $z_0 \Lo_{n+1}^1 = Q_0 \Lo_{n+1}^1$.
The diffeomorphism $\bQ_{z_0}: z_0 \Lo_{n+1}^1 \to \cU_{z_0}$
is defined as follows.
Given $M \in z_0 \Lo_{n+1}^1$ write $M = QR$,
$Q \in \SO_{n+1}$, $R \in \Up_{n+1}^{+}$
(that is, $R$ is upper triangular with positive diagonal).
There exists a unique $z \in \cU_{z_0}$ with $\Pi(z) = Q$:
define $\bQ_{z_0}(M) = z$.
Its inverse $\bL_{z_0}: \cU_{z_0} \to  z_0 \Lo_{n+1}^1$
can be defined by $\bL_{z_0}(z) = z_0 \bL(z_0^{-1} z)$.
The maps $\bQ_{z_0}: z_0 \Lo_{n+1}^1 \to \cU_{z_0}$ and 
its inverse $\bL_{z_0}$ also respect the Bruhat stratification.
A curve $\Gamma: J \to z_0 \Lo_{n+1}^1$
is convex if and only if
$z_0^{-1} \Gamma: J \to \Lo_{n+1}^1$ is convex
(as in Equation~\eqref{equation:convexLo}).

\begin{example}
\label{example:z0z1z2}
Take $n = 3$, 
$\sigma_0 = bacb$, $\sigma_1 = aba$ and $\sigma_2 = bcb$.
Define $\rho_i = \eta \sigma_i$
so that $\rho_0 = ac$, $\rho_1 = abc$ and $\rho_2 = cba$
so that $\rho_0 \vartriangleleft \rho_1$ and
$\rho_0 \vartriangleleft \rho_2$.
Take $q_0 = -1 \in \Quat_4$.
Consider
\[
z_0 = q_0 \acute\rho_0 = -\acute a\acute c \in \Spin_4, \qquad
Q_0 = \Pi(z_0) = \begin{pmatrix}
0 & -1 & 0 & 0 \\
1 & 0 & 0 & 0 \\
0 & 0 & 0 & -1 \\
0 & 0 & 1 & 0 \end{pmatrix} \in \Bru_{\rho_0} \subset \SO_4, \]
$z_1 = q_0 \acute\rho_1 = -\acute a\acute b\acute c$ and
$z_2 = q_0 \acute\rho_2 = -\acute c\acute b\acute a$.
The coset $z_0 \Lo_4^1$ is the set of matrices
\begin{equation}
\label{equation:z0L}
z_0 \Lo_4^1 = Q_0 \Lo_4^1 =
\left \{ 
\begin{pmatrix}
y_1 & -1 & 0 & 0 \\
1 & 0 & 0 & 0 \\
x_1 & x_2 & y_2 & -1 \\
x_3 & x_4 & 1 & 0 \end{pmatrix}; \; x_1, x_2, x_3, x_4, y_1, y_2 \in \RR 
\right\}.
\end{equation}
The preimage
$\bQ_{z_0}^{-1}[\Bru_{z_0}]$
is the two dimensional affine space of matrices defined by
$(x_1=x_2=x_3=x_4=0)$ (with $y_1$ and $y_2$ free).
We shall also study the sets 
$\bQ_{z_0}^{-1}[\Bru_{z_1}],
\bQ_{z_0}^{-1}[\Bru_{z_2}] \subset z_0 \Lo^1_4$.
\end{example}


For $t > 0$, let $D_{\ray, t} = \Diag(1,t,t^2,\ldots,t^n)$.
For $M = z_0 L \in z_0 \Lo_{n+1}^1$ let
$\ray(t,M) = z_0 D_{\ray, t} L D^{-1}_{\ray,t} \in z_0 \Lo_{n+1}^1$.
Thus, for $M$ as in Equation~\eqref{equation:z0L} we have
\begin{equation}
\label{equation:rays}
\ray(t,M) =
\begin{pmatrix}
t y_1 & -1 & 0 & 0 \\
1 & 0 & 0 & 0 \\
t^3 x_1 & t^2 x_2 & t y_2 & -1 \\
t^2 x_3 & t x_4 & 1 & 0 \end{pmatrix}. 
\end{equation}
Notice that $\ray(t,M)$ is Bruhat-equivalent to $M$
and that $\lim_{t\to 0} \ray(t,M) = z_0$.
For fixed $M$, the set of matrices $\ray(t,M)$, $t > 0$,
is the \textit{ray} through $M$.

Following \cite{Goulart-Saldanha0} (Equation (4)), define the subgroups
\[ \Up_\rho = \{ U \in \Up^1_{n+1} \;|\;
((i < j) \land (U_{ij} \ne 0)) \to (i,j) \in \Inv(\rho) \} \le \Up^1_{n+1} \]
and $\Lo_\rho = (\Up_\rho)^\top \le \Lo^1_{n+1}$;
$\Inv(\rho)$ is the set of inversions of $\rho$
(as in Equation~\eqref{equation:inv}).

\begin{remark}
\label{remark:LoLo}
If $z_0 = q_0 \acute\rho_0 \in  \tilde\B_{n+1}^{+}$ then
$\bQ_{z_0}^{-1}[\Bru_{z_0}] = z_0 \Lo_{\rho_0^{-1}} = \Up_{\rho_0} z_0$.
Also, if $\rho_0 = \eta\sigma_0$ then
$\Inv(\sigma_0^{-1}) \cap \Inv(\rho_0^{-1}) = \varnothing$,
$\Inv(\sigma_0^{-1}) \cup \Inv(\rho_0^{-1}) = \Inv(\eta)$ and
$z_0 \Lo_{\sigma_0^{-1}} = \Lo_{\sigma_0} z_0$.
\end{remark}

\begin{lemma}
\label{lemma:LoLo}
If $\rho = \eta\sigma \in S_{n+1}$ then
the map below is a diffeomorphism:
\[ \Lo_{\rho^{-1}} \times \Lo_{\sigma^{-1}} \to \Lo_{n+1}^1, \qquad
(L_\bfy,L_\bfx) \mapsto L_\bfy L_\bfx. \]
\end{lemma}

\begin{proof}
For this proof, order the entries of $L \in \Lo_{n+1}^1$
in increasing order of $|i-j|$:
the diagonal entries come first, next the entries of the form $(j+1,j)$
and so on.

Fill in the off-diagonal entries of $L_\bfy \in \Lo_{\rho^{-1}}$
and $L_\bfx \in \Lo_{\sigma^{-1}}$ with free variables
$y_{ij}, x_{ij} \in \RR$.
Expand the product $L_\bfy L_\bfx$
to obtain polynomials $P_{ij}$ in the variables $y_{i,\ast}$ and $x_{\ast,j}$.
For $i > j$, each $P_{ij}$ has a single monomial of degree $1$
(either $y_{ij}$ or $x_{ij}$)
and a number of monomials of degree two in earlier variables.
Thus, in order to invert the map it suffices to 
follow the above order in the entries.
\end{proof}

\begin{example}
\label{example:z0z1z2w}
As in Example~\ref{example:z0z1z2},
take $n = 3$, $\rho_0 = ac$
and $\sigma_0 = bacb$.
We have $\Inv(\rho_0^{-1}) = \{(1,2),(3,4)\}$,
$\Inv(\sigma_0^{-1}) = \{(1,3),(1,4),(2,3),(2,4)\}$,
so that $\Lo_{\rho_0^{-1}}$ and $\Lo_{\sigma_0^{-1}}$
are the sets of matrices of the forms
\[ L_\bfy = \begin{pmatrix}
1 & & & \\ -y_1 & 1 & & \\ & & 1 & \\ & & -y_2 & 1 
\end{pmatrix} \in \Lo_{\rho_0^{-1}}, \qquad
L_\bfx = \begin{pmatrix}
1 & & & \\ & 1 & & \\ x_3 & x_4 & 1 & \\ -\tilde x_1 & -\tilde x_2 & & 1 
\end{pmatrix} \in \Lo_{\sigma_0^{-1}} \]
where $\tilde x_1,\tilde x_2,x_3,x_4,y_1,y_2 \in \RR$.
If we take $x_1 = \tilde x_1 + x_3y_2$ and $x_2 = \tilde x_2 + x_4y_2$
we have that $M = Q_0 L_\bfy L_\bfx =
z_0 L_\bfy L_\bfx \in z_0 \Lo_{4}^1$ is as in
Example~\ref{example:z0z1z2}.
Comparing with the notation in the proof of Lemma~\ref{lemma:LoLo},
we use different indices and sometimes apply a minus sign
for compatibility with
Example~\ref{example:z0z1z2}.

From Remark~\ref{remark:LoLo},
$z_0 \Lo_{\rho_0^{-1}} = \Up_{\rho_0} z_0$.
From Lemma~\ref{lemma:LoLo} or from direct computation,
every matrix $M \in z_0 \Lo_{4}^1$ can be uniquely written
as $M = z_0 L_\bfy L_\bfx = U_\bfy z_0 L_\bfx$
for $L_\bfy \in \Lo_{\rho_0^{-1}}$, $U_\bfy \in \Up_{\rho_0}$
and $L_\bfx \in \Lo_{\sigma_0^{-1}}$.
It follows that $z_0 L_{\bfy,0} L_\bfx$ and $z_0 L_{\bfy,1} L_\bfx$
are Bruhat equivalent
(if $L_{\bfy,\ast} \in \Lo_{\rho_0^{-1}}$ 
and $L_\bfx \in \Lo_{\sigma_0^{-1}}$).
The section
$z_0 \Lo_{\sigma_0^{-1}} \subset z_0 \Lo_{4}^1$
already has the essential information
concerning Bruhat cells.

In the notation of our examples, $z_0 \Lo_{\sigma_0^{-1}}$ is the subset
of matrices of the form
\[ 
\tilde M =
\begin{pmatrix}
0 & -1 & 0 & 0 \\
1 & 0 & 0 & 0 \\
\tilde x_1 & \tilde x_2 & 0 & -1 \\
x_3 & x_4 & 1 & 0 \end{pmatrix}
\in z_0 \Lo_{\sigma_0^{-1}} \subset z_0 \Lo_4^1.
\]
Let $\rho_1, \rho_2 \in S_4$ and $z_1, z_2 \in \tilde\B^{+}_4$
be as in Example~\ref{example:z0z1z2}.
The preimages under $\bQ_{z_0}: z_0 \Lo_{n+1}^1 \to \cU_{z_0}$ of the subsets
$\Bru_{\rho_0}, \cU_{z_0} \cap \Bru_{\rho_1}, \cU_{z_0} \cap \Bru_{\rho_2}
\subset \, \cU_{z_0}$
are characterized by
$(\tilde x_1 = \tilde x_2 = x_3 = x_4 = 0)$,
$(\tilde x_1 = x_3 = x_4 = 0, \tilde x_2 \ne 0)$ and
$(\tilde x_1 = \tilde x_2 = x_4 = 0, x_3 \ne 0)$, respectively
(with $y_1$ and $y_2$ free in all cases).
Similarly, the preimages of
$\cU_{z_0} \cap \Bru_{z_1}$ and $\cU_{z_0} \cap \Bru_{z_2}$
are characterized by
$(\tilde x_1 = x_3 = x_4 = 0, \tilde x_2 > 0)$ and
$(\tilde x_1 = \tilde x_2 = x_4 = 0, x_3 > 0)$, respectively.
The intersection of these two preimages with $z_0 \Lo_{\sigma_0^{-1}}$
are two rays in $z_0 \Lo_{n+1}^1$.
\end{example}

Our next aim is to define the subsets 
\[ \Pos_{z_0,\rho}, \Neg_{z_0,\rho} \subseteq z_0 \Lo_{n+1}^1 \]
and their counterparts in $\cU_{z_0} \subset \Spin_{n+1}$,
\[
\Pos^{\Spin}_{z_0,\rho} = \bQ_{z_0}[\Pos_{z_0,\rho}],
\quad
\Neg^{\Spin}_{z_0,\rho} = \bQ_{z_0}[\Neg_{z_0,\rho}].
\]
We begin with special cases; the general case takes a little longer.
If $z_0 = q_0 \acute\rho_0$ (with $q_0 \in \Quat_{n+1}$)
and $\rho \not\ge \rho_0$ (in the strong Bruhat order)
then $\Pos_{z_0,\rho} = \Neg_{z_0,\rho} = \varnothing$;
this is consistent with the fact that
$\Bru_{\rho} \cap \,\cU_{z_0} = \varnothing$.
Set
$\Pos_{z_0,\rho_0} = \Neg_{z_0,\rho_0}
= z_0 \Lo_{\rho_0^{-1}} = \Up_{\rho_0} z_0 = \bQ_{z_0}^{-1}[\Bru_{z_0}]$
and
$\Pos^{\Spin}_{z_0,\rho_0} = \Neg^{\Spin}_{z_0,\rho_0} = \Bru_{z_0}$.

\begin{lemma}
\label{lemma:pos}
Consider $z_0 = q_0 \acute\rho_0 \in \tilde\B_{n+1}^{+}$.
For $M \in  z_0 \Lo_{n+1}^1$, the following conditions are equivalent:
\begin{enumerate}
\item{There exists
$\tilde M \in z_0 \Lo_{\rho_0^{-1}}$
and $\tilde L \in \Pos_\eta \subset \Lo_{n+1}^1$
such that $M = \tilde M\tilde L$.}
\item{There exists a convex curve $\Gamma: [0,1] \to z_0 \Lo_{n+1}^1$
with 
\[ \Gamma(0) \in z_0 \Lo_{\rho_0^{-1}}, \qquad \Gamma(1) = M.\]}
\item{There exists a convex curve $\Gamma: [0,1] \to \cU_{z_0}$
with 
\[ \Gamma(0) \in \Bru_{z_0}, \qquad \Gamma(1) = \bQ_{z_0}(M).\]}
\end{enumerate}
\end{lemma}


\begin{proof}
The equivalence between items (2) and (3) follows
by applying $\bQ_{z_0}$ or $\bL_{z_0}$
and from the definition of convexity for $\Gamma: [0,1] \to z_0 \Lo_{n+1}^1$.

Assume that item (2) holds and
consider a convex curve $\Gamma: [0,1] \to z_0 \Lo_{n+1}^1$
with $\tilde M = \Gamma(0) \in z_0 \Lo_{\rho_0^{-1}}$ and $M = \Gamma(1)$.
As in Equation \eqref{equation:positiveconvex},
we have $\tilde L = \tilde M^{-1} M \in \Pos_{\eta} \subset \Lo_{n+1}^1$,
proving item (1).

Conversely, assume that item (1) holds. There is a convex curve 
$\Gamma_0: [0,1] \to \Lo_{n+1}^1$ with
$\Gamma_0(0) = I$ and $\Gamma_0(1) = \tilde L$.
Take $\Gamma(t) = \tilde M \Gamma_0(t)$ to prove item (2).
\end{proof}


We are now ready to define $\Pos_{z_0,\eta}$:
a matrix
$M \in z_0 \Lo_{n+1}^{1}$ belongs to $\Pos_{z_0,\eta}$
if~and~only~if it satisfies
one of the conditions in Lemma~\ref{lemma:pos}.

Similarly, 
$M \in z_0 \Lo_{n+1}^{1}$ belongs to $\Neg_{z_0,\eta}$
if and only if
there exists a convex curve $\Gamma: [-1,0] \to z_0 \Lo_{n+1}^1$
with $\Gamma(-1) = M$ and $\Gamma(0) \in z_0 \Lo_{\rho_0^{-1}}$.
Set 
\[
\Pos^{\Spin}_{z_0,\eta} = \bQ_{z_0}[\Pos_{z_0,\eta}] \subseteq \cU_{z_0}, 
\qquad
\Neg^{\Spin}_{z_0,\eta} = \bQ_{z_0}[\Neg_{z_0,\eta}] \subseteq \cU_{z_0}. 
\]
We have $\adv(z_0) = q_0 \acute\eta$
(from the algebraic formula for $\adv$);
Lemma~\ref{lemma:pos}
(with the definition of $\adv$ in Equation~\eqref{equation:advchop})
implies that 
$\Pos^{\Spin}_{z_0,\eta} \subseteq \Bru_{\adv(z_0)}$.
We therefore have
\[ \Pos^{\Spin}_{z_0,\eta} \subseteq \cU_{z_0} \cap \Bru_{q_0\acute\eta}. \]
For very small values of $n$ this is sometimes an equality
(see Example~\ref{example:z0z1z2x}).

For $z = q\acute\rho$ with $q \in \Quat_{n+1}$ and $\rho \ge \rho_0$ define
\[
\BL_{z_0,z} =
\bQ_{z_0}^{-1}[\Bru_{z} \cap \,\cU_{z_0}] =
\bL_{z_0}[\Bru_{z} \cap \,\cU_{z_0}]
\subseteq  z_0 \Lo_{n+1}^{1}, 
\]
which is either empty or
a submanifold of dimension $\inv(\rho)$.
For large $n$, most values of $z_0$ and $z$ with $\rho \ge \rho_0$,
the sets $\BL_{z_0,z}$ have many connected components
and these are possibly complicated.
We use and generalize the notation of \cite{Alves-Saldanha2},
where $\BL_{z}$ denotes our
$\BL_{1,z} \subseteq \Lo_{n+1}^{1}$.

We are finally ready to define $\Pos_{z_0,\rho}$ 
for the remaining cases $\rho \ge \rho_0$: set
\[
\begin{aligned}
\Pos_{z_0,\rho} &=
\overline{\Pos_{z_0,\eta}} \cap
\BL_{z_0,q_0 \acute\rho} 
\subseteq  z_0 \Lo_{n+1}^{1}, \\
\Pos^{\Spin}_{z_0,\rho} &=
\bL_{z_0}^{-1}[\Pos_{z_0,\rho}] =
\cU_{z_0} \cap \overline{\Pos^{\Spin}_{z_0,\eta}} \cap \Bru_{q_0 \acute\rho}
\subseteq \cU_{z_0}.
\end{aligned}
\]
Notice that
$\Pos_{z_0,\rho} =
\bQ_{z_0}^{-1}[\Pos^{\Spin}_{z_0,\rho}]
\subseteq  z_0 \Lo_{n+1}^{1}$;
the following proposition lists a few basic facts about these sets.

\begin{prop}
\label{prop:Posz0}
Let $z_0 = q_0 \acute\rho_0$, 
$\rho_0 \in S_{n+1}$, $\rho_0 \ne \eta$ and $q_0 \in \Quat_{n+1}$.
We have
\[ \partial\Pos_{z_0,\eta} =
\bigsqcup_{\rho_0 \le \rho \ne \eta} \Pos_{z_0,\rho}
\subseteq z_0 \Lo_{n+1}^1. \]
For $\rho \ge \rho_0$ the subset
$\Pos_{z_0,\rho} \subseteq z_0 \Lo_{n+1}^1$
is nonempty, connected and a submanifold of dimension $\inv(\rho)$.
\end{prop}

\begin{proof}
Assume for simplicity that $q_0 = 1$.
As we already saw,
the sets $z_0 \Lo_{n+1}^1$ and $\cU_{z_0}$ are equivalent
via the maps $\bQ_{z_0}$ and $\bL_{z_0}$,
that is, these maps respect the Bruhat stratification.

A neighborhood $U_0 \subset z_0 \Lo_{n+1}^1$ of $z_0$ 
is \textit{star shaped}
if $M \in U_0$ and $0 < t \le 1$ imply $\ray(t,M) \in U_0$.
There are arbitrarily small star shaped neighborhoods $U_0$ of $z_0$.
Since rays respect the Bruhat stratification,
we may as well restrict our attention to such small neighborhoods $U_0$
or to $U_1 = \bQ_{z_0}[U_0] \subset \cU_{z_0}$.

Via a projective transformation
(see \cite{Goulart-Saldanha0,Goulart-Saldanha1}),
$z_0$ and $U_1$ are moved to
$z_2 \in \bQ[\Pos_{\rho_0}] \subseteq \cU_{1} \cap \Bru_{z_0}$ and 
$U_2 \subset \cU_{1} \cap \cU_{z_0}$, respectively.
The Bruhat stratification is again respected.
The map $\bL$ now takes $z_2$ and $U_2$
to $L_3 \in \Pos_{\rho_0}$ and $U_3 \subset \Lo_{n+1}^1$.
Thus, the Bruhat stratification of $z_0 \Lo_{n+1}^1$
is equivalent to the Bruhat stratification
of a small neighborhood $U_3 \subset \Lo_{n+1}^1$
of a point $L_3 \in \Pos_{\rho_0}$.

The desired sets $\Pos_{z_0,\rho}$ correspond to the intersections
$\Pos_{\rho} \cap\,U_3$ and are therefore submanifolds
of the correct dimension.
The fact that these intersections are nonempty and connected
follows from the rays and the remarks
after Equation~\eqref{equation:boundarypos}.
\end{proof}

\begin{example}
\label{example:z0z1z2x}
Take $n = 3$,
$\rho_0 = ac = ca$, $\rho_1 = abc$, $\rho_2 = cba$,
$z_0 = -\acute a\acute c = -\acute c\acute a$,
$z_1 = -\acute a\acute b\acute c$
and $z_2 = -\acute c\acute b\acute a$,
as in
Examples~\ref{example:z0z1z2} and \ref{example:z0z1z2w}. 
We have 
$\chop(z_0) = \chop(z_1) = \chop(z_2) = \acute\eta$ and
$\adv(z_0) = \adv(z_1) = \adv(z_2) = -\acute\eta$
(see Equation~\eqref{equation:advchop}
and Theorem 3 in \cite{Goulart-Saldanha0}).
Notice that $\Bru_{z_0} \subset \overline{\Bru_{z_i}}$
for $i = 1$ and $i = 2$.

If a
convex curve $\Gamma: [-1,1] \to \cU_{z_0}$
satisfies
$\Gamma(0) \in \Bru_{z_0} \cup \Bru_{z_1} \cup \Bru_{z_2}$
we then have $\Gamma(t) \in \Bru_{\acute\eta}$ for $t < 0$
and $\Gamma(t) \in \Bru_{-\acute\eta}$ for $t > 0$
(see Equation \eqref{equation:advchop}).
A computation shows that
\[ \Pos_{z_0,\rho_1} = \Neg_{z_0,\rho_1}, \qquad
\Pos_{z_0,\rho_2} = \Neg_{z_0,\rho_2} \]
are the preimages under $\bQ_{z_0}$ 
of $\cU_{z_0} \cap \Bru_{z_1}$ and $\cU_{z_0} \cap \Bru_{z_2}$,
described at the end of Example~\ref{example:z0z1z2w}.
In this example we also have
the rather exceptional equalities:
\[ \Pos_{z_0,\rho_1} = \BL_{z_0,z_1}, \qquad
\Pos_{z_0,\rho_2} = \BL_{z_0,z_2}. \]

For $t > 0$, $t$ small,
we have
$-\acute a \alpha_2(t) \acute c \in \Pos^{\Spin}_{z_0,\rho_1}$ and
$-\acute c \alpha_2(t) \acute a \in \Pos^{\Spin}_{z_0,\rho_2}$
(see Theorem~1 in~\cite{Goulart-Saldanha0}).
\end{example}

\bigbreak

\section{Closed subsets and neighborhoods}
\label{section:closed}

Recall that $M^\ast$ is the set of words using only letters from $M$,
the subset of $S_{\PA} < S_4$ defined in  Equation \eqref{equation:MY}.
We need to write explicit words $w \in M^\ast$ often,
and the notation $(aba,cba)$ is somewhat cumbersome.
We therefore often write instead $[aba][cba]$.
Here, parentheses and commas are omitted.
Inside each pair of brackets we have a reduced word
(such as $aba$ or $cba$) for an element of $M \subset S_4$;
this element is a letter in the word $w \in M^\ast$.

We first prove that each subset
$\cM_{\mu_0,\mu_1} \subset \cL_3$ is closed;
the function $\mu$ and the sets $\cM_{\mu_0,\mu_1}$ 
are defined in Section~\ref{section:context},
above Proposition~\ref{prop:M}.
We need a few lemmas;
see Equation~\eqref{equation:sqsubseteq}
for the definition of the relation $\sqsubseteq$.

\begin{lemma}
\label{lemma:sqsubseteq}
If $w \in M^\ast$ and $\sigma \in S_{4} \smallsetminus \{e\}$
satisfy $w \sqsubseteq \sigma$
then either $w = \sigma$ (so that $w$ is a word of length $1$) or $(w,\sigma)$
is one of the following:
\[ ([aba],bacb), \quad ([bcb],bacb), \quad
([aba][cba],\eta), \quad ([cba][aba],\eta). \]
In all cases, $\sigma \in M$ and $\mu(w) = \mu(\sigma)$.
\end{lemma}

\begin{proof}
The first remark is that $w \in M^\ast$
implies $\hat w \in Z(\Quat_4)$ which then implies
$\hat\sigma \in Z(\Quat_4)$ and therefore
(by brute force or by Lemma 2.2 in \cite{Goulart-Saldanha-cw})
we have
$\sigma \in S_{\PA} \smallsetminus \{e\} = M \cup \{ac, abc\}$.
For each such $\sigma$, we list all possible words $w \in M^\ast$
with $w \sqsubseteq \sigma$:
this is a computation, and not a very long one.
\end{proof}

\begin{lemma}
\label{lemma:sqsubseteq2}
Consider $w_0 \in M^\ast$ and $w_1 \in \Word_3$.
If $w_0 \sqsubseteq w_1$ then $w_1 \in M^\ast$ and $\mu(w_0) = \mu(w_1)$.
\end{lemma}

\begin{proof}
Write $w_1 = \sigma_1\cdots\sigma_\ell$:
it follows from the definition of $\sqsubseteq$
that $w_0$ is a concatenation $w_{1,0}\cdots w_{\ell,0}$
of words $w_{i,0} \in M^\ast$ with $w_{i,0} \sqsubseteq \sigma_i$.
It follows from Lemma \ref{lemma:sqsubseteq} that
$\sigma_i \in M$ and $\mu(w_{i,0}) = \mu(\sigma_i)$ (for all $i$).
The result follows by addition:
$\mu(w_0) = \sum_i \mu(w_{i,0}) = \sum_i \mu(\sigma_i) = \mu(w_1)$.
\end{proof}

\begin{prop}
\label{prop:closed}
For any $(\mu_0,\mu_1) \in \NN^2$,
the subset $\cM_{\mu_0,\mu_1} \subset \cL_3$ is closed.
\end{prop}

\begin{proof}
Consider $w_0 \in M^\ast$ with $\mu(w_0) = (\mu_0,\mu_1)$
and $w_1 \in \Word_3$.
If $\cL_3[w_1] \cap \overline{\cL_3[w_0]} \ne \varnothing$,
Equation (10) in \cite{Goulart-Saldanha-cw} tells us that
$w_0 \sqsubseteq w_1$.
Lemma \ref{lemma:sqsubseteq2} then implies
$w_1 \in M^\ast$ and $\mu(w_0) = \mu(w_1)$.
We therefore have $\cL_3[w_1] \subseteq \cM_{\mu_0,\mu_1}$.
Thus, $\overline{\cL_3[w_0]} \subseteq \cM_{\mu_0,\mu_1}$.
Since this holds for every $w_0 \in M^\ast$ with $\mu(w_0) = (\mu_0,\mu_1)$
it follows that $\cM_{\mu_0,\mu_1}$ is closed.
\end{proof}

The proof that each closed subset $\cM_{\mu_0,\mu_1}$
is a topological submanifold is longer,
and will require that we study neighborhoods of curves
$\Gamma_0 \in \cM_{\mu_0,\mu_1}$.
We are particularly interested in detecting
other curves in $\cM$ in such neighborhoods.

\begin{remark}
\label{remark:neighborhood}
Given a curve $\Gamma_0$ with known itinerary,
what can we say about possible itineraries of neighboring curves $\Gamma$?
That is not an easy question,
and is related to the Shapiro-Shapiro conjecture
\cite{Saldanha-Shapiro-Shapiro,Saldanha-Shapiro-Shapiro1}.
The partial orders $\preceq$ and $\sqsubseteq$
give us some incomplete but useful information.

Recall that $\Bru_\eta \subset \Spin_4$ is an open dense subset.
Given $\Gamma_0 \in \cL_3$, $\Gamma_0: [0,1] \to \Spin_4$,
let $\sing(\Gamma_0) = \{t_1 < \cdots < t_\ell\} \subset (0,1)$
be its \textit{singular set}, so that, for $t \in (0,1)$,
\[ t \in \sing(\Gamma_0) \quad\iff\quad \Gamma_0(t) \notin \Bru_\eta. \]
Assume that $\Gamma_0(t_k) \in \Bru_{z_k} \subset \Bru_{\rho_k}$,
$z_k \in \acute\rho_k\Quat_{4} \subset \widetilde\B_{4}^{+}$,
$\rho_k = \eta\sigma_k \in S_4$,
so that the itinerary of $\Gamma_0$ is
$\sigma_1\cdots\sigma_\ell \in \Word_4$.
Given $\sing(\Gamma_0)$, take $\epsilon_0 > 0$ with
$8\epsilon_0 < \min_k |t_k - t_{k-1}|$
and consider the intervals
\[ 
J_k = [t_k-\epsilon_0,t_k+\epsilon_0], \qquad
\tilde J_k = \begin{cases}
[0,t_1-\epsilon_0], & k = 1, \\
[t_{k-1}+\epsilon_0,t_k-\epsilon_0], & 1 < k \le \ell, \\
[t_\ell+\epsilon_0,1], & k = \ell + 1. \end{cases}
\]
The restriction $\Gamma_0|_{\tilde J_k}$ has image contained in
$\Bru_{\chop(z_k)} = \Bru_{\adv(z_{k-1})}$
where $\chop(z_k) = \adv(z_{k-1})
\in \acute\eta\Quat_4 \subset \widetilde\B_{4}^{+}$,
so that $\Bru_{\chop(z_k)}$ is a connected component of $\Bru_\eta$.
The restriction $\Gamma_0|_{\tilde J_k}$ is therefore convex
(Lemma 2.6 in \cite{Goulart-Saldanha1})
with empty itinerary;
if $\Gamma: \tilde J_k \to \Bru_\eta$
is near the restriction of $\Gamma_0$
then its itinerary is also empty.
For sufficiently small $\epsilon_0 > 0$,
the restriction $\Gamma_0|_{J_k}$ has image contained in $\cU_{z_k}$:
recall that $\cU_{z_k} = z_k\,\cU_1 \subset \Spin_4$
is an open neighborhood of $z_k$ containing $\Bru_{z_k}$.
The restriction $\Gamma_0|_{J_k}$ is therefore also convex,
and its itinerary is $\sigma_k$;
there are usually many possible itineraries for
$\Gamma: J_k \to \Bru_\eta$ near the restriction of $\Gamma_0$,
but all such possible itineraries are non empty
(by Theorem 1 in \cite{Goulart-Saldanha1}).

Let $W_{k,+} \subset \cU_{z_k} \cap \Bru_{\adv(z_{k})}$
and $W_{k,-} \subset \cU_{z_k} \cap \Bru_{\chop(z_{k})}$
be small contractible neighborhoods
of $\Gamma_0(t_k+\epsilon_0)$ and $\Gamma_0(t_k-\epsilon_0)$, respectively.
The sets $W_{k,\pm}$ can be chosen sufficiently small such that
for all $z_{k-1,+} \in W_{k-1,+}$ and $z_{k,-} \in W_{k,-}$
there exist convex curves $\Gamma: \tilde J_k \to \Bru_{\chop(z_k)}$
with $\Gamma(t_{k-1}+\epsilon_0) = z_{k-1,+}$
and $\Gamma(t_k-\epsilon_0) = z_{k,-}$.
There exists a contractible neighborhood $\tilde U_k$ of 
the restriction $\Gamma_0|_{\tilde J_k}$ in $\cL_3(\cdot,\cdot)$
and a homeomorphism $\phi_k: \tilde U_k \to W_{k-1,+} \times W_{k,-} \times H$
such that, for all $\Gamma \in \tilde U_k$,
\[ \phi_k(\Gamma) = (\Gamma(t_{k-1}+\epsilon_0),\Gamma(t_k-\epsilon_0),\ast); \]
here $H$ denotes a separable infinite dimensional Hilbert space.
There are similar statements for initial and final jets
(see Section 6 in \cite{Goulart-Saldanha1},
including Lemma 6.7).

Thus, in order to study itineraries 
in a neighborhood of $\Gamma_0$
it suffices to study itineraries in neighborhoods
of the restrictions $\Gamma_0|_{J_k}$.
More precisely, there exist neighborhoods
$U \subset \cL_3$ of $\Gamma_0$
and $U_k \subset \cL_3(\cdot,\cdot)$
of the restrictions $\Gamma_0|_{J_k}$ in $\cL_3$
with the following properties.
For all $\Gamma \in U$,
the singular set of $\Gamma$ is contained in $\bigcup_k J_k$
and therefore the itinerary of $\Gamma$ is the concatenation
of the itineraries of the restrictions $\Gamma|_{J_k}$
(again by Theorem 1 in \cite{Goulart-Saldanha1}).
The map $\phi$ defined in $U$ 
taking $\Gamma$ to 
$(\Gamma|_{J_1}, \ldots, \Gamma|_{J_\ell})$
assumes values in
$U_1 \times \cdots \times U_\ell$.
There exists a homeomorphism
$\Phi: U \to (U_1 \times \cdots \times U_\ell) \times H$
with $\Phi(\Gamma) = (\phi(\Gamma), \ast)$.

This procedure is closely related to the constructions
in \cite{Goulart-Saldanha1, Goulart-Saldanha-cw},
particularly in the study of the partial orders
$\preceq$, $\sqsubseteq$ and of
the boundaries of cells $c_\sigma$ in the CW complex $\cD_3$.
\end{remark}

In the present section we consider the easier cases
$\sigma \in \{aba,bcb,cba\}$.

Figure \ref{fig:ababcb} show us the boundaries
of the cells $c_{[aba]}$ and $c_{[bcb]}$.
These boundaries have been extensively studied
in \cite{Goulart-Saldanha1} and \cite{Goulart-Saldanha-cw}.
We focus on $[aba]$; the case $[bcb]$ is similar.

\begin{figure}[ht]
\def\svgwidth{12cm}
\centerline{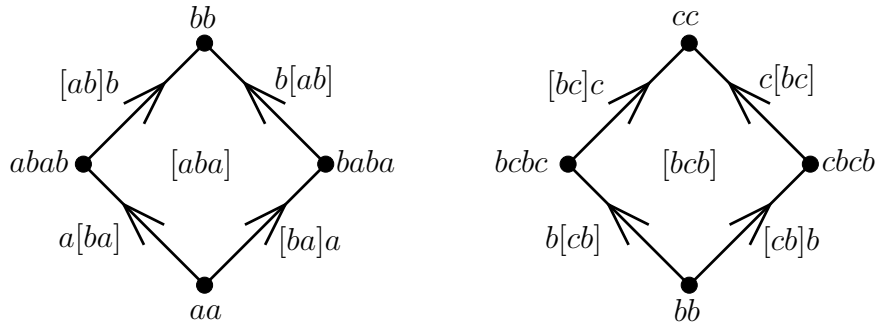}
\caption{The cells $c_{[aba]}$ and $c_{[bcb]}$.}
\label{fig:ababcb}
\end{figure}

The only words $w \in \Word_3$ with $w \sqsubseteq [aba]$
are the $9$ words shown in Figure \ref{fig:ababcb};
reading cyclically, the new $8$ words are:
\[ aa, a[ba], abab, [ab]b, bb, b[ab], baba, [ba]a. \]
Clearly, each one contains a letter of dimension $0$
and is therefore not in $M^\ast$. 
If $\Gamma \in \cL_3$ has itinerary $[aba]$
then in small neighborhood of $\Gamma$
all curves in $\cM$ also have itinerary $[aba]$.
In particular, near $\Gamma$ the subset $\cM_{1,0} \subset \cL_3$
is a smooth submanifold of codimension $2$
(see Equation \eqref{equation:dimw}).
Also, the intersection of $\cM$ with this neighborhood 
is contained in $\cM_{1,0}$.


\begin{figure}[ht]
\def\svgwidth{80mm}
\centerline{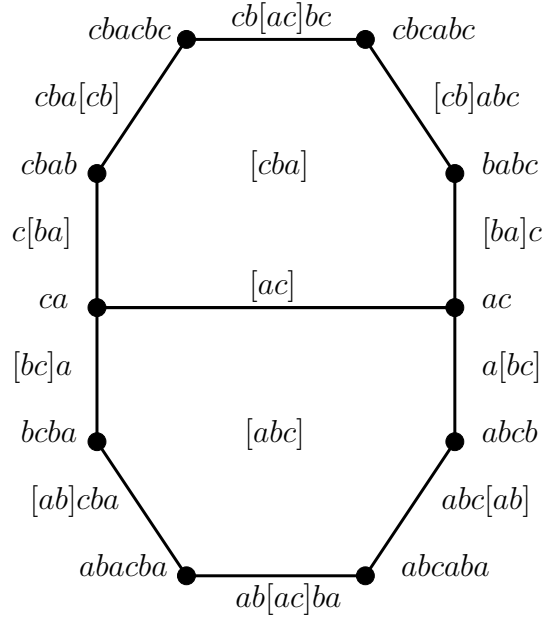}
\caption{The cells $c_{[abc]}$ and $c_{[cba]}$.}
\label{fig:cbaabc}
\end{figure}

Figure \ref{fig:cbaabc}
shows us the boundaries of $c_{[cba]}$ and of its twin $c_{[abc]}$.
Notice that the edge $c_{[ac]}$ is common to both boundaries.
The only words $w \in \Word_3$ with $w \sqsubseteq [cba]$
are the $13$ shown in the upper half of Figure \ref{fig:cbaabc}.
Besides $[cba]$ itself, they can be read cyclically as
\[ ac, [ac], ca, c[ba], cbab, cba[cb], cbacbc, cb[ac]bc,
cbcabc, [cb]abc, babc, [ba]c. \]
Each of these words except $[ac]$ contains a letter of dimension $0$
and is therefore not in $M^\ast$. 
Also, by construction,
we have $[ac] \notin M$,
taking care of this last case.
Thus, if $\Gamma \in \cL_3$ has itinerary $[cba]$
then in small neighborhood of $\Gamma$
all curves in $\cM$ also have itinerary $[cba]$.
As in the previous case,
near $\Gamma$ the subset $\cM_{0,1} \subset \cL_3$
is a smooth submanifold of codimension $2$ 
(see again Equation \eqref{equation:dimw}) 
and the intersection of $\cM$ with this neighborhood 
is contained in $\cM_{0,1}$.

Clearly, the only word $w \in M^\ast$ 
with $\mu(w) = (0,1)$ is $w = [cba]$:
we thus have $\cM_{0,1} = \cL_3[[cba]]$.
A short computation verifies that $[cba] \sqsubseteq \sigma$
implies $\sigma = cba$:
thus, $\cL_3[[cba]] \subset \cL_3$ is a closed subset.
It follows that $\cM_{0,1} \subset \cL_3$ is a closed subset
and a submanifold of codimension $2$,
consistently with Proposition \ref{prop:M}.


\begin{remark}
\label{remark:abcaccba}
Following \cite{Goulart-Saldanha-cw},
let $\Ideal_{[0]} \subset \Word_3$
be the set of words with at least one letter of dimension $0$.
This set is discussed in Proposition 11.3 in \cite{Goulart-Saldanha-cw}.
The set $\Ideal_{[0]}$ is contained in the larger set
$\Ideal_{Y_2}$ of words containing at least one letter
in $Y_2 = \{a,b,c,ba,ab,bc\} \subset S_4$.
The set $\Ideal_{Y_2}$ is discussed
in Proposition 13.2 in \cite{Goulart-Saldanha-cw}.

The boundaries of $c_{[aba]}$ or $c_{[bcb]}$
are formed by cells with labels in $\Ideal_{[0]}$
(see Figure~\ref{fig:ababcb}).
The boundaries of $c_{[abc]}$ or $c_{[cba]}$ are not.
However, as shown in Figure~\ref{fig:cbaabc},
the union $c_{[abc]} \cup c_{[ac]} \cup c_{[cba]}$
is a 2-disk and its boundary is also formed
by cells with labels in $\Ideal_{[0]}$.
\end{remark}

\begin{remark}
\label{remark:abc}
Notice that in \cite{Goulart-Saldanha-cw}
we consider the set $Y = S_{n+1} \smallsetminus S_{\PA}$
before introducing the larger set $\tilde Y$.
This construction is under some points of view more natural
and would give us, instead of $M$, the larger set
\[ M_{[ac]} = S_{\PA} \smallsetminus \{e\} =
M \cup \{ac,abc\} = \{aba,bacb,bcb,ac,abc,cba,\eta\}. \]
Let $M_{[ac]}^\ast \subset \Word_3$ be the set of words
using only letters in $M_{[ac]}$.
Let $\cM_{[ac]} \supset \cM$ be the set of locally convex curves
with itineraries in $M_{[ac]}^\ast$.
The subset $\cM_{[ac]} \subset \cL_3$
shares several interesting properties with $\cM$,
with certain significant adaptations being necessary.
For instance,
$\cM_{[ac],0,1} = \cL_3[[abc]] \cup \cL_3[[ac]] \cup \cL_3[[cba]]$
is a closed subset of $\cL_3$ and a connected component of $\cM_{[ac]}$.
The subset $\cM_{[ac],0,1} \subset \cL_3$ is not,
however, a topological submanifold:
it is a submanifold with boundary
(the boundary being $\cL_3[[abc]] \cup \cL_3[[cba]]$).

On the other hand, the set
\[ M_{[abc]} =
(M \smallsetminus \{cba\}) \cup \{abc\} =
\{aba,bacb,bcb,abc,\eta\} \]
is an alternative to $M$.
In this case most results and proofs
apply with minor changes.
For instance, let $\cM_{[abc]} \subset \cL_3$
be the set of locally convex curves with itineraries in $M_{[abc]}^\ast$.
The stratum $\cL_3[[abc]] = \cM_{[abc],0,1} \subset \cL_3$
is a connected component of $\cM_{[abc]}$,
a closed subset of $\cL_3$
and a submanifold of codimension $2$.
The other connected components of $\cM_{[abc]}$
are likewise closed subsets of $\cL_3$
and topological submanifolds of even codimension.
\end{remark}

\bigbreak

\section{$bacb$}
\label{section:bacb}

The case $[bacb]$ has higher dimension.
Still, a list of all words $w \in \Word_3$
such that $w \sqsubseteq [bacb]$ can be worked out.
In this case there are, besides $w = [bacb]$,
two such words in $M^\ast$:
these are $[aba]$ and $[bcb]$;
other important words $w$ satisfying $w \sqsubseteq [bacb]$
are $bc[aba]cb$ and $ba[bcb]ab$.
Figure \ref{fig:Xbacb} sketches a neighborhood of 
$\cM$ near a curve $\Gamma_0$ of itinerary $[bacb]$.
A more careful discussion of this example is in order.

\begin{figure}[ht]
\def\svgwidth{50mm}
\centerline{
\begingroup%
  \makeatletter%
  \providecommand\color[2][]{%
    \errmessage{(Inkscape) Color is used for the text in Inkscape, but the package 'color.sty' is not loaded}%
    \renewcommand\color[2][]{}%
  }%
  \providecommand\transparent[1]{%
    \errmessage{(Inkscape) Transparency is used (non-zero) for the text in Inkscape, but the package 'transparent.sty' is not loaded}%
    \renewcommand\transparent[1]{}%
  }%
  \providecommand\rotatebox[2]{#2}%
  \newcommand*\fsize{\dimexpr\f@size pt\relax}%
  \newcommand*\lineheight[1]{\fontsize{\fsize}{#1\fsize}\selectfont}%
  \ifx\svgwidth\undefined%
    \setlength{\unitlength}{310bp}%
    \ifx\svgscale\undefined%
      \relax%
    \else%
      \setlength{\unitlength}{\unitlength * \real{\svgscale}}%
    \fi%
  \else%
    \setlength{\unitlength}{\svgwidth}%
  \fi%
  \global\let\svgwidth\undefined%
  \global\let\svgscale\undefined%
  \makeatother%
  \begin{picture}(1,0.56451613)%
    \lineheight{1}%
    \setlength\tabcolsep{0pt}%
    \put(0,0){\includegraphics[width=\unitlength,page=1]{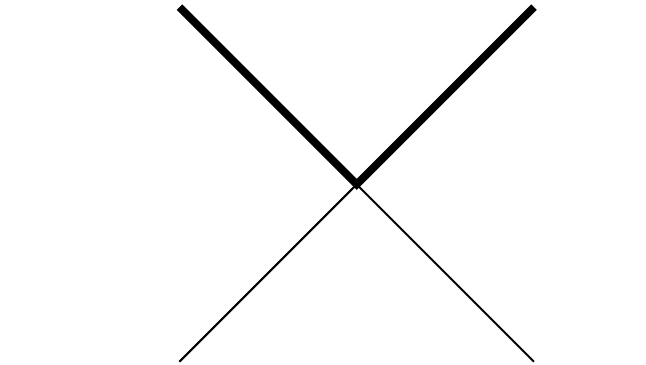}}%
    \put(0.18632346,0.37051455){\color[rgb]{0,0,0}\makebox(0,0)[lt]{\lineheight{1.25}\smash{\begin{tabular}[t]{l}$[aba]$\end{tabular}}}}%
    \put(0.73538452,0.37051455){\color[rgb]{0,0,0}\makebox(0,0)[lt]{\lineheight{1.25}\smash{\begin{tabular}[t]{l}$[bcb]$\end{tabular}}}}%
    \put(0.0033031,0.09598402){\color[rgb]{0,0,0}\makebox(0,0)[lt]{\lineheight{1.25}\smash{\begin{tabular}[t]{l}$ba[bcb]ab$\end{tabular}}}}%
    \put(0.62557231,0.26070234){\color[rgb]{0,0,0}\makebox(0,0)[lt]{\lineheight{1.25}\smash{\begin{tabular}[t]{l}$[bacb]$\end{tabular}}}}%
    \put(0.79029063,0.09598402){\color[rgb]{0,0,0}\makebox(0,0)[lt]{\lineheight{1.25}\smash{\begin{tabular}[t]{l}$bc[aba]cb$\end{tabular}}}}%
  \end{picture}%
\endgroup%
}
\caption{A transversal section to $\cL_3[[bacb]]$.
Since $\cL_3[[bacb]] \subset \cL_3$ has codimension $3$,
the X figure should be seen in a space of dimension $3$.
The thicker lines show $\cL_3[[aba]]$ and $\cL_3[[bcb]]$,
which are also contained in $\cM$.
The thinner lines are smooth continuations,
but are not in $\cM$.}
\label{fig:Xbacb}
\end{figure}

As in Example~\ref{example:z0z1z2} and the other examples
in Section~\ref{section:positivity},
let $\sigma_0 = bacb$, $\sigma_1 = aba$ and $\sigma_2 = bcb$.
Define $\rho_i = \eta \sigma_i$
so that $\rho_0 = ac$, $\rho_1 = abc$ and $\rho_2 = cba$.
Take $q_0 = -1 \in \Quat_4$.

Let $\cL_3[[bacb]] \subset \cL_3$
be the stratum of curves with itinerary $[bacb]$.
We follow the method discussed in Section 7
of \cite{Goulart-Saldanha1}
to study a transversal section to $\cL_3[[bacb]]$.
Such a family of curves is parametrized by
$(x_1,x_2,x_3)$  in a neighborhood of $0 \in \RR^3$ by
\[ M = \begin{pmatrix} t & 1 & 0 & 0 \\ 1 & 0 & 0 & 0 \\
\frac{t^3}{6} + x_2 t + x_1 & \frac{t^2}{2} + x_2 & t & 1 \\
\frac{t^2}{2} + x_3 & t & 1 & 0 \end{pmatrix} \]
so that $t \mapsto M$ is a family (parametrized by $(x_1,x_2,x_3)$)
of convex curves in $P_{bacb} \Lo_{4}^{1}$
($P_{bacb} \in \B_4^{+}$ is a permutation matrix).
For a family $(\Gamma_{(x_1,x_2,x_3)})$ of convex curves in the spin group,
apply Gram-Schmidt to $M$ to obtain the orthogonal matrix
$\Pi(\Gamma_{(x_1,x_2,x_3)}(t)) \in \SO_4$.
In order to study the itinerary of $\Gamma_{(x_1,x_2,x_3)}$
(as a function of $(x_1,x_2,x_3)$),
let $m_a, m_b, m_c$ be the determinants of
the lower left minors of orders $1,2,3$, respectively:
\[ m_a = \frac{t^2}{2} + x_3, \qquad
m_b = -\frac{t^4}{12} + \frac{(x_2-x_3)t^2}{2} + x_1t - x_2x_3, \qquad
m_c = -\frac{t^2}{2} + x_2. \]
The singular set of $\Gamma_{(x_1,x_2,x_3)}$
is the set of roots of the product $m_am_bm_c$.
In order to investigate multiple roots, set
\begin{gather*}
q_{[ab]} = x_3, \qquad q_{[ac]} = (x_2+x_3)^2, \qquad q_{[cb]} = x_2, \\
q_{[ba]} = 18 x_2^2 x_3 - 12 x_2 x_3^3 + 2 x_3 + 9 x_1^2, \qquad
q_{[bc]} = 18 x_2 x_3^2 - 12 x_2^2 x_3 + 2 x_2 + 9 x_1^2
\end{gather*}
so that we have (by explicit computation)
\begin{gather*}
\resultant(m_a,m_b,t) = \frac{1}{72}\;q_{[ab]} q_{[ba]}, \quad
\resultant(m_b,m_c,t) = \frac{1}{72}\;q_{[bc]} q_{[cb]}, \\
\discriminant(m_a,t) = -2\,q_{[ab]}, \quad
\discriminant(m_c,t) = 2\,q_{[cb]}, \\
\resultant(m_a,m_c,t) = \frac{1}{4}\;q_{[ac]}, \quad
\discriminant(m_b,t) = \frac{1}{432}\;q_{[ba]} q_{[bc]}.
\end{gather*}
These polynomials are constructed so that, for instance,
the curve $\Gamma_{(x_1,x_2,x_3)}$
has a letter $[ab]$ in its itinerary if and only if $q_{[ab]} = 0$.
Consider the unit sphere $\Ss^2 \subset \RR^3$
in coordinates $(x_1,x_2,x_3)$.
The polynomials $q_\ast$ define curves in $\Ss^2$;
in the complementary open regions the itinerary of $\Gamma$
has dimension $0$
(i.e., the itinerary
has only the basic letters $a$, $b$ and $c$):
see Figure~\ref{fig:bacb-2}.
The point at the left where
the vertical black line (corresponding to $q_{[ab]}$)
is tangent to the red curve (corresponding to $q_{[ba]}$)
marks the point where the itinerary is $[aba]$.
The antipodal point,
where the ``other'' black line is also tangent to the red curve
corresponds to itinerary $bc[aba]cb$.
Similarly, the point at the right where
the vertical magenta line (corresponding to $q_{[cb]}$)
is tangent to the blue curve (corresponding to $q_{[bc]}$)
marks the point where the itinerary is $[bcb]$;
the antipodal point corresponds to itinerary $ba[bcb]ab$.

\begin{figure}[ht]
\def\svgwidth{\linewidth}
\centerline{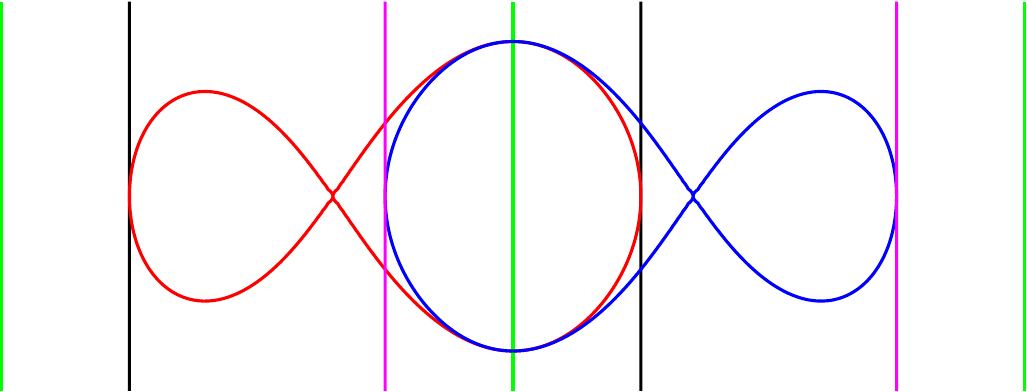}
\caption{A transversal section to $\cL_3[[bacb]]$.
The $2$-sphere is drawn like the Earth in a Mercator map.
The north and south poles are $\pm e_1$,
with $-e_1$ above and $e_1$ below.
The equator is a horizontal line in the middle
(not drawn) and corresponds to $x_1 = 0$,
with the points $e_2$, $-e_3$, $-e_2$ and $e_3$
reading from left to right.
Zero sets of $q_{[ab]}$, $q_{[ba]}$, $q_{[ac]}$, $q_{[bc]}$ and $q_{[cb]}$
are shown in black, red, green,  blue and magenta, respectively.
The polynomial $q_{[ac]}$ is a square
and the green lines are actually double.
The one at the center has the effect of swapping two pairs
$ac \leftrightarrow ca$;
the one appearing twice at both ends of the figure
detects the coincidence of complex roots and therefore swaps nothing
in terms of itineraries or real roots.
This and similar examples are joint work with Boris Shapiro.
The figure was produced with the help of the Maple software.}
\label{fig:bacb-2}
\end{figure}

This transversal section is in fact not generic.
This is the same phenomenon discussed at length
for $[acb]$ in Section 9 of \cite{Goulart-Saldanha1}.
A small perturbation obtains a valid boundary for $c_{[bacb]}$,
as in Figure \ref{fig:bacb};
there are other valid boundaries.

\begin{figure}[ht]
\centerline{\includegraphics[width =14cm]{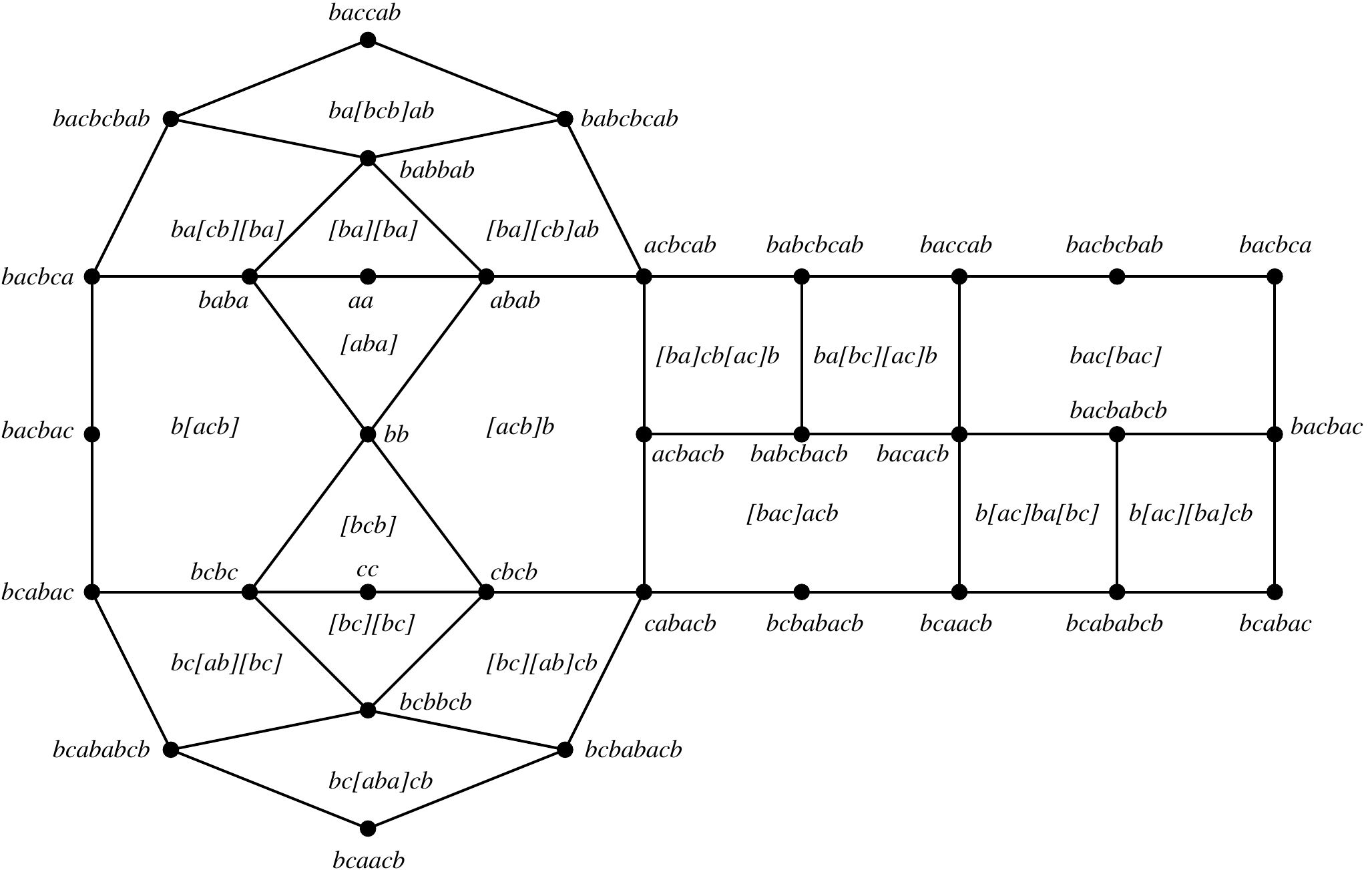}}
\caption{A valid boundary for $[bacb]$.}
\label{fig:bacb}
\end{figure}

\begin{lemma}
\label{lemma:cM10}
Let $\sigma_0 = bacb$, $\rho_0 = \eta\sigma_0 = ac$
and $z_0 \in \acute\rho_0 \Quat_4 \subset \widetilde\B_4^{+}$.
Consider a short convex curve
$\Gamma_0: [-\epsilon_0,\epsilon_0] \to \cU_{z_0}$
with itinerary $[bacb] = \sigma_0$.
Let $\cM_\ast$ be the set of short convex curves
$\Gamma:  [-\epsilon_0,\epsilon_0] \to \cU_{z_0}$
with one of the itineraries: $[bacb]$, $[aba]$ or $[bcb]$.
Then $\cM_{\ast}$ is a topological submanifold of codimension $2$
in a neighborhood of $\Gamma_0$ in $\cL_3(\cdot,\cdot)$.
\end{lemma}

\begin{coro}
\label{coro:cM10}
The set 
\[ \cM_{1,0} = \cL_3[[aba]] \cup \cL_3[[bacb]] \cup \cL_3[[bcb]] \]
is a closed subset of $\cL_3$ and a topological submanifold
of codimension $2$.
\end{coro}

\begin{proof}
The difference between the statement in the corollary
and a case of the statement in the lemma
is that in the corollary two extra convex arcs
are drawn from $1$ to $\Gamma(-\epsilon_0)$
and from $\Gamma(\epsilon_0)$ to $-1$.
Such arcs can be chosen so as not to influence the statement,
as discussed in Remark \ref{remark:neighborhood}
and in \cite{Goulart-Saldanha1}.
\end{proof}


\begin{proof}[Proof of Lemma \ref{lemma:cM10}]
This proof is based on the examples discussed in
Section~\ref{section:positivity}.
We may assume without loss of generality that $z_0 = -\acute a\acute c$
(as in the examples).
As discussed above in the proof of Corollary \ref{coro:cM10},
it is equivalent to discuss a neighborhood of a short convex curve
$\Gamma: [-\epsilon_0,\epsilon_0] \to \cU_{z_0}$
with $\Gamma(0) = z_0$
or of a curve $\tilde\Gamma \in \cL_3[[bacb]]$.
In the proof we switch between the two points of view.

We saw in \cite{Goulart-Saldanha1}
(and in Section \ref{section:closed})
that $\cL_3[[aba]] \subset \cL_3$ and $\cL_3[[bcb]] \subset \cL_3$
are submanifolds of codimension $2$.
We are left with proving that if $\Gamma \in \cL_3[[bacb]]$
then $\cM_{\ast} \subset \cL_3(\cdot,\cdot)$ is
a topological submanifold of codimension $2$
in a neighborhood of $\Gamma$.
We prefer to argue in the situation of Corollary \ref{coro:cM10}.

Consider a curve $\Gamma_0 \in \cL_3[[bacb]]$
with $\Gamma_0(t_0) \in \Bru_{z_0}$.
After a projective transformation, 
we may assume that $\Gamma_0(t_0) = z_0$.
For nearby curves $\Gamma$,
there is a continuously defined $t_0(\Gamma)$
for which the $x_4$ coordinate of
$M(\Gamma) = \bQ^{-1}(\Gamma(t_0(\Gamma))) \in z_0 \Lo_4^1$
is equal to $0$.
Let $x(\Gamma) = (x_1(\Gamma), x_2(\Gamma), x_3(\Gamma))$
be the respective coordinates of $M(\Gamma)$.
In a neighborhood $U$ of $\Gamma_0$,
the map $x: U \to \RR^3$ is a topological submersion.
The submanifold $\cL_3[[bacb]]$ is defined by $x(\cdot) = 0$.
Let $v: \RR^3 \to \RR^2$ 
\[ v(x_1,x_2,x_3) = \left(x_1,x_2+x_3-|x_2-x_3|\right). \]
Notice that $v(x_1,x_2,x_3) = 0$ precisely 
at the origin and on
the half lines $(x_1=x_2=0,x_3>0)$ and $(x_1=x_3=0,x_2>0)$.
Thus, $v \circ x: U \to \RR^2$ is a topological submersion,
and $(v\circ x)(\Gamma) = 0$ if and only if $\Gamma \in \cM_{1,0}$.
\end{proof}

\bigbreak

\section{$\eta$, $[aba][cba]$, $[cba][aba]$}
\label{section:eta}

We have $\eta = abacba$ and
\[ [aba][cba] \sqsubseteq \eta, \qquad [cba][aba] \sqsubseteq \eta; \]
furthermore, together with $\eta$ itself,
these are the only words $w \in M^\ast$ with $w \sqsubseteq \eta$.
The subset $\cM_{1,1} \subset \cL_3$ equals
\[ \cM_{1,1} =
\cL_3[[aba][cba]] \sqcup \cL_3[\eta] \sqcup \cL_3[[cba][aba]] \subset \cL_3. \]
The subsets
\[ \cL_3[\eta], \cL_3[[aba][cba]], \cL_3[[cba][aba]] \subset \cL_3 \]
are contractible submanifolds of codimensions $5$, $4$ and $4$,
respectively.
The aim of this section is to prove that they glue nicely.

\begin{lemma}
\label{lemma:eta}
Let $\sigma_0 = \eta$, $\rho_0 = \eta\sigma_0 = e$
and $z_0 \in \Quat_4 \subset \widetilde\B_4^{+}$.
Consider a short convex curve
$\Gamma_0: [-\epsilon_0,\epsilon_0] \to \cU_{z_0}$
with itinerary $\eta = \sigma_0$.
Let $\cM_\ast$ be the set of short convex curves
$\Gamma:  [-\epsilon_0,\epsilon_0] \to \cU_{z_0}$
with one of the itineraries: $\eta$, $[aba][cba]$ or $[cba][aba]$.
Then $\cM_{\ast}$ is a submanifold of codimension $4$
in a neighborhood of $\Gamma_0$ in $\cL_3(\cdot,\cdot)$.
\end{lemma}

\begin{coro}
\label{coro:eta}
The subset $\cM_{1,1} \subset \cL_3$ is a closed subset
and a submanifold of codimension $4$.
\end{coro}

\begin{proof}
The proof of this corollary is similar to that of Corollary \ref{coro:cM10}.
\end{proof}

\begin{remark}
\label{remark:etaC1}
The careful reader will have noticed that
Lemma \ref{lemma:eta} and Corollary \ref{coro:eta}
claim that $\cM_{1,1}$
is a submanifold, not a topological submanifold.
Indeed, we will prove that it is at least a $C^1$ submanifold.
Moreover, for any $k \in \NN$, $k > 0$,
if we take the space of curves to be $H^r$,
the subset $\cM_{1,1} \subset \cL_3$ is a $C^k$ submanifold
provided we take $r$ to be sufficiently large.
\end{remark}


\begin{proof}[Proof of Lemma \ref{lemma:eta}]
Set $\sigma_0 = \eta$, $\sigma_1 = aba$ and $\sigma_2 = cba$.
Set $\rho_i = \eta\sigma_i$ so that
$\rho_0 = e$, $\rho_1 = abc$ and $\rho_2 = bcb$.
Consider the open set $\cU_1 \subset \Spin_4$
and the usual bijection $\bQ: \Lo_4^1 \to \cU_1$
which respects Bruhat cells.
We denote the entries of $L \in \Lo_4^1$
or of a convex curve $\Gamma: J \to \Lo_4^1$
(where $J \subseteq \RR$ is an interval) as:
\[ L =
\begin{pmatrix} 
1 & 0 & 0 & 0 \\
l_{21} & 1 & 0 & 0 \\
l_{31} & l_{32} & 1 & 0 \\
l_{41} & l_{42} & l_{43} & 1 
\end{pmatrix},
\qquad
\Gamma(t) =
\begin{pmatrix} 
1 & 0 & 0 & 0 \\
g_{21}(t) & 1 & 0 & 0 \\
g_{31}(t) & g_{32}(t) & 1 & 0 \\
g_{41}(t) & g_{42}(t) & g_{43}(t) & 1 
\end{pmatrix}.
\]
For a convex curve $\Gamma: J \to \Lo_4^1$,
we have positive functions
$\beta_1, \beta_2, \beta_3$ with
\[ \Gamma'(t) = \Gamma(t) \Lambda(t), \qquad
\Lambda(t) =
\begin{pmatrix} 
0 & 0 & 0 & 0 \\
\beta_{1}(t) & 0 & 0 & 0 \\
0 & \beta_{2}(t) & 0 & 0 \\
0 & 0 & \beta_{3}(t) & 0 
\end{pmatrix}.
\]
In particular, we have $g'_{21}(t) = \beta_1(t)$
so that there exists $t_0 \in \RR \supset J$ with
$\sign(g_{21}(t)) = \sign(t-t_0)$.
The value $t_0$ is a sufficiently smooth function of $\Gamma$.
For simplicity, we assume from now on that $t_0 = 0$ is in the interior of $J$.
Let $x_{ij} = g_{ij}(0)$.
The map taking $\Gamma$ to $(x_{31},x_{32},x_{41},x_{42},x_{43})$
is a submersion from an open neighborhood $U \subset \cL_3$
of $\cL_3[\eta]$ to $\RR^5$.
Recall that $\Gamma \in \cL_3[\eta]$ if and only if
$(x_{31} = x_{32} = x_{41} = x_{42} = x_{43} = 0)$,
confirming that $\cL_3[\eta]$ is a submanifold of codimension $5$.

The closure of $\bQ^{-1}[\Bru_{\rho_2}] \subset \Lo_4^1$
is the subgroup of matrices $L$ with
$(l_{21} = l_{31} = l_{41} = 0)$.
More precisely, $L$ belongs to this subgroup
if and only if $L \in \bQ^{-1}[\Bru_\rho]$ for some $\rho \le \rho_2$
(in the strong Bruhat order).
Let $N_2 \subset U$ be 
the submanifold of $U$ of codimension $2$
defined by $(x_{31} = x_{41} = 0)$.
It follows that $\Gamma \in N_2$
if and only if
$\Gamma(0) \in \bQ^{-1}[\Bru_\rho]$ for some $\rho \le \rho_2$.
When we convert this into information about 
the itinerary $\iti(\Gamma)$,
we learn that if $\Gamma \in N_2$
then $\iti(\Gamma)$ contains a letter $\sigma \ge \sigma_2$.
The converse also holds, giving us a characterization of $N_2$.
We restrict our attention to this submanifold.

The function $g_{32}$ is strictly increasing
since it satisfies 
\begin{equation}
\label{equation:g32}
g'_{32}(t) = \beta_2(t) > 0, \qquad g_{32}(0) = x_{32}, \qquad
g_{32}(t) = x_{32} + \int_0^t \beta_2(\tau) d\tau.
\end{equation}
Consider $t_{32} \in J$ such that 
$\sign(g_{32}(t)) = \sign(t-t_{32})$:
clearly, $\sign(t_{32}) = -\sign(x_{32})$.
The function $g_{31}$ satisfies
$g'_{31}(t) = \beta_1(t) g_{32}(t)$:
it is therefore strictly decreasing for $t < t_{32}$
and strictly increasing for $t > t_{32}$.
If $\Gamma \in N_2$ (so that $g_{31}(0) = 0$) and $t_{32} > 0$,
there exists a unique $t_{31} > t_{32}$ such that $g_{31}(t_{31}) = 0$.
Similarly,
if $\Gamma \in N_2$ and $t_{32} < 0$,
there exists a unique $t_{31} < t_{32}$ such that $g_{31}(t_{31}) = 0$.
If $\Gamma \in N_2$ and $t_{32} = 0$, set $t_{31} = 0$
so that $t_{31}$ is defined for any $\Gamma \in N_2$
and it always satisfies $\sign(t_{31}) = -\sign(x_{32})$.
Notice that the function $g_{32}$ has a unique real zero at $t_{32}$:
in particular, $x_{32} \ne 0$ implies $g_{32}(t_{31}) \ne 0$.
We claim that $t_{31}$ is a sufficiently smooth function of $x_{32}$,
$\beta_1$ and $\beta_2$.

Indeed, $g_{32}$ is defined by Equation \eqref{equation:g32}
and is therefore a sufficiently smooth function of $x_{32}$ and $\beta_2$.
Set
\[ h_{31}(t) =
\int_0^1 \beta_1(st) g_{32}(st) ds =
\begin{cases}
\dfrac{g_{31}(t)}{t}, & t \ne 0, \\
g'_{31}(0), & t = 0,
\end{cases} \]
a sufficiently smooth function of $g_{32}$ and $\beta_1$
(we assume $\Gamma \in N_2$ and therefore $g_{31}(0) = 0$).
The function $h_{31}$ has a unique zero at $t_{31}$.
We also have $h'_{31}(t_{31}) \ne 0$.
Indeed, since $g_{31}(t) = t\,h_{31}(t)$ we have
$g'_{31}(t) = t\,h'_{31}(t) + h_{31}(t) = \beta_1(t) g_{32}(t)$.
If $x_{32} \ne 0$ we have
$t_{31} h'_{31}(t_{31}) = \beta_1(t_{31}) g_{32}(t_{31}) \ne 0$.
If $x_{32} = 0$ we have $g_{32}(0) = 0$, $g'_{32}(0) \ne 0$ so that
$g_{31}(0) = g'_{31}(0) = 0$, $g''_{31}(0) \ne 0$ and therefore
$h_{31}(0) = 0$,
$h'_{31}(0) = \lim_{t\to 0} g_{31}(t)/t^2 \ne 0$.
The claim that $t_{31}$ is sufficiently smooth
follows from the implicit function theorem.

Assume $x_{32}$ given, so that $t_{31}$ is also determined:
we construct $z_2$ and $z_3$,
which are sufficiently smooth functions of $x_{32}$
and of the functions $\beta_i$.
Given $z$ (a temporary and auxiliary real variable),
define functions $g_{43,z}$ and $g_{42,z}$ by
\begin{align*}
g_{43,z}(t) &= \int_0^t \beta_3(\tau)d\tau + z = g_{43,0}(t) + z, \\
g_{42,z}(t) &= \int_{t_{31}}^t \beta_2(\tau) g_{43,z}(\tau) d\tau =
g_{42,0}(t) + z \int_{t_{31}}^t \beta_2(\tau) d\tau.
\end{align*}
Write $\tilde z_2(z) = g_{42,z}(0)$, 
a sufficiently smooth function of $x_{32}$ and $z$.
Finally, define the function $g_{41,z}$ by
\[ g_{41,z}(t) = \int_{t_{31}}^t \beta_1(\tau) g_{42,z}(\tau) d\tau =
g_{41,0}(t) + z \int_{t_{31}}^t \int_{t_{31}}^\tau
\beta_1(\tau) \beta_2(\tilde\tau) d\tilde\tau d\tau. \]
We want to write $g_{41} = g_{41,z}$ for an appropriate $z = z_3$
(recall that $x_{41} = g_{41}(0)$ is supposed to be $0$);
this requires a few more computations.
For $x_{32} = 0$, we have $t_{31} = 0$
so that $z_2 = g_{42,z}(0) = 0$ and $g_{41,z}(0) = 0$ for all $z$:
in this case set $z_3 = 0$.
In general, 
we have that
\[ g_{41,z}(0) = g_{41,0}(0) + Hz, \quad
H = \int_{t_{31}}^0 \int_{t_{31}}^\tau
\beta_1(\tau) \beta_2(\tilde\tau) d\tilde\tau d\tau. \]
Near $x_{32} = t_{31} = 0$, we have
\[ g_{41,0}(0) = C t_{31}^3 + O(t_{31}^4), \qquad
H = \tilde C t_{31}^2 + O(t_{31}^3) \]
where $C$ and $\tilde C$ are positive constants
(or, more precisely, sufficiently smooth positive functions
of the $\beta_i$).
For $t_{31} \ne 0$,
we can therefore solve $g_{41,z}(0) = 0$
to obtain a unique solution $z = z_3$;
for $t_{31}$ near $0$ we have
$z_3 = (-C/\tilde C) t_{31} + O(t_{31}^2)$.
Thus, both $z_3$ and $z_2 = \tilde z_2(z_3)$
are $C^1$ functions of $x_{32}$
(and the functions $\beta_i$),
including at $x_{32} = 0$.

Define the function $\phi: N_2 \to \RR^2$
by $\phi(\Gamma) = (x_{43}-z_3,x_{42}-z_2)$.
This function is clearly a submersion.
We claim that $\Gamma \in \cM_{1,1}$
if and only if $\phi(\Gamma) = (0,0)$:
the proof of this claim completes the proof of the lemma.

We first prove that if $\phi(\Gamma) = (0,0)$ then $\Gamma \in \cM_{1,1}$.
Indeed, if $x_{32} = 0$ we have $\Gamma(0) = I$
and the itinerary of $\Gamma$ is $\eta$.
If $x_{32} \ne 0$, there are at least two 
moments of non transversality: $t = 0$ and $t = t_{31} \ne 0$.
We have $\Gamma(0) \in \bQ^{-1}[\Bru_\rho]$
for some $\rho \le \rho_2$.  
The closure of $\bQ^{-1}[\Bru_{\rho_1}] \subset \Lo_4^1$
is the subset of matrices $L$ with
$(l_{31} = l_{41} = l_{42} = 0)$.
Thus we have $\Gamma(t_{31}) \in \bQ^{-1}[\Bru_\rho]$
for some $\rho \le \rho_1$.
It is possible to prove directly that if $\Gamma \in N_2$
satisfies $\phi(\Gamma) = (0,0)$ then its itinerary
is either $\eta$, $[aba][cba]$ or $[cba][aba]$;
it is probably easier to invoke Theorem 4 in \cite{Goulart-Saldanha1}
(extra letters would violate multiplicity).

Conversely, assume $\Gamma \in \cM_{1,1}$.
There is a unique value of $t$ for which $g_{21} = 0$:
this value corresponds to a letter $\sigma_2 = [cba]$ or $\sigma_0 = \eta$
in the itinerary of $\Gamma$.
Reparametrize $\Gamma$ so that $g_{21}(0) = 0$,
$\Gamma \in N_2$ and the above discussion applies.
If $\iti(\Gamma) = \eta$ we have $\Gamma(0) = I$.
Otherwise $\Gamma(0) \in \bQ^{-1}[\Bru_{\rho_2}]$;
$t_{31}$ is such that $\Gamma(t_{31}) \in \bQ^{-1}[\Bru_{\rho_1}]$.
This implies that $x_{43} = z_3$ and $x_{42} = z_2$
and therefore $\phi(\Gamma) = (0,0)$, as desired.
\end{proof}

Recall that we know from Proposition \ref{prop:closed}
that  $\cM_{\mu_0,\mu_1} \subset \cL_3$ is a closed subset.

\bigbreak

We conclude this section by summing up
our conclusions as a proposition.

\begin{prop}
\label{prop:submanifold}
For any pair $(\mu_0,\mu_1) \in \NN^2$,
the subset $\cM_{\mu_0,\mu_1} \subset \cL_3$
is a non-empty topological submanifold of codimension $2(\mu_0+\mu_1)$.
Moreover, there exists a transversal bundle
of vector spaces of dimension $2(\mu_0+\mu_1)$.
\end{prop}

\begin{proof}
The first claim follows from Lemmas \ref{lemma:cM10} and \ref{lemma:eta}
and Remark \ref{remark:neighborhood}.
The second claim (concerning the transversal bundle)
follows from the fact that $\cM_{\mu_0,\mu_1}$ is mostly smooth 
and well behaved where it is not.
\end{proof}




\bigbreak

\section{Contractibility of $\cM_{\mu_0,\mu_1}$}
\label{section:contractibility}

At this point we know that $\cM_{\mu_0,\mu_1} \subset \cL_3$
is a closed subset and a topological submanifold
of codimension $2(\mu_0+\mu_1)$.
The set $\cM_{\mu_0,\mu_1}$ is
stratified as
\[ \cM_{\mu_0,\mu_1} =
\bigsqcup_{w \in M^\ast,\;\mu(w) = (\mu_0,\mu_1)}
\cL_3[w] \subset \cL_3. \]
The aim of this section is to prove that $\cM_{\mu_0,\mu_1}$
is contractible.
We first need a lemma about topological Hilbert manifolds.
The result follows easily from, for instance, \cite{Palais};
we include a proof for completeness and comodity of the reader.

\begin{lemma}
\label{lemma:3contractible}
Let $\cH$ be a topological Hilbert manifold.
Assume that there exists a topological submanifold $\cH_0 \subset \cH$
with the following properties.
\begin{enumerate}
\item{The subset $\cH_0 \subset \cH$ is closed and contractible.}
\item{The submanifold $\cH_0 \subset \cH$ has codimension $1$
and admits an open tubular neighborhood $U_0$,
$\cH_0 \subset U_0 \subseteq \cH$.}
\item{The complement $\cH \smallsetminus \cH_0$
has two connected components $\cH_{+}$ and $\cH_{-}$,
both open subsets of $\cH$, both contractible.}
\end{enumerate}
Then $\cH$ is also contractible.
\end{lemma}

\begin{proof}
Let $U_0$ be the tubular neighborhood as in the statement.
Let $\Phi: \cH_0 \times (-1,1) \to U_0$ be a homeomorphism
with $\Phi(x,0) = x$ for all $x \in \cH_0$.
We may assume without loss of generality that
$\Phi^{-1}[U_0 \cap \cH_{+}] = \cH_0 \times (0,1)$.
Let $U_{\pm} = \cH_{\pm} \cup U_0$.
The inclusions $\cH_{\pm} \subset U_{\pm}$
are deformation retracts 
(explicitly constructed via $\Phi$)
and therefore 
the open subsets $U_{\pm} \subset M$ are both contractible.
The intersection $U_0 = U_{+} \cap U_{-}$ is also contractible
(again from $\Phi$).
The theorems of Seifert--Van Kampen and Mayer--Vietoris
imply that $\cH$ is connected, path connected,
simply connected,
with all homology groups $H_k(\cH)$ trivial for $k \ge 1$.
Thus, the homotopy groups $\pi_k(\cH)$, $k \ge 1$, are also trivial.
Thus, the inclusion $\cH_{+} \subset \cH$ is a weak homotopy equivalence.
Since both $\cH_{+}$ and $\cH$ are topological Hilbert manifolds
the inclusion is a homotopy equivalence
and $\cH$ is contractible (\cite{Palais}, Corollary of Theorem 15),
as desired.
\end{proof}

Before stating our next lemma
we need a few auxiliary definitions.
Define an equivalence relation in $M$ by
$\sigma_0 \sim \sigma_1$ if and only if $\mu(\sigma_0) = \mu(\sigma_1)$.
Equivalently, the nontrivial equivalences are
$aba \sim bacb \sim bcb$.
Let $M^\bullet = \{ aba, cba, \eta \} \subset M$
so that $M^\bullet$ has a unique representative
of each $\sim$-equivalence class in $M$.
Extend the equivalence relation $\sim$ to $M^\ast$:
two words are equivalent if they have the same length
and are letter by letter equivalent.
The set of words $(M^\bullet)^\ast \subset M^\ast$
has a unique representative of each equivalence class.
For $w_0 \in (M^\bullet)^\ast$, we have
$\mu(w_0) = (N_{aba}+N_\eta, N_{cba} + N_\eta)$
where $N_\sigma$ is the number of copies of $\sigma$ in $w_0$.
The equivalence class $\{w_1 \in M^\ast, w_1 \sim w_0\}$
has $3^{N_{aba}}$ elements.

For $w_0 \in (M^\bullet)^\ast$, let
\begin{equation}
\label{equation:Lbullet}
\cL_3^\bullet[w_0] =
\bigsqcup_{w_1 \in M^\ast, w_1 \sim w_0} \cL_3[w_1] \subseteq
\cM_{\mu(w_0)} \subset \cL_3. 
\end{equation}
Thus, if $\mu(w_0) = (\mu_0,\mu_1)$ and $N_\eta = 0$ then
$\cL_3^\bullet[w_0] \subset \cM_{\mu_0,\mu_1}$
is an open subset.
The subset $\cL_3^\bullet[w_0] \subset \cL_3$
is a topological Hilbert submanifold
of codimension $2(\mu_0 + \mu_1) + N_\eta$
but usually not a closed subset.

\begin{lemma}
\label{lemma:bulletcontractible}
For $w_0 \in (M^\bullet)^\ast$, 
the subset $\cL_3^\bullet[w_0] \subset \cM_{\mu_0,\mu_1}$
is contractible.
\end{lemma}

\begin{proof}
Let $N = N_{aba}$ be the number of copies of $[aba]$ in $w_0$.
Let $I_N = \{w_0\}$ and
$I_0 = \{ w_1 \in M^\ast, w_1 \sim w_0\}$.
For $0 < k < N$, let $I_k \subset I_0$
be the set of words for which the first $k$ letters
in $\{aba,bacb,bcb\}$ are equal to $aba$
so that $|I_k| = 3^{(N-k)}$.
For $k > 0$ and $w \in I_k$,
we define $A_{k,+}(w), A_{k,0}(w), A_{k,-}(w) \in I_{k-1}$.
Take $A_{k,+}(w) = w \in I_{k-1}$.
The word $A_{k,0}(w) \in I_{k-1}$ is obtained from $w$
by substituting $[bacb]$ for the $k$-th copy of $[aba]$ in $w$.
Similarly,
the word $A_{k,-}(w) \in I_{k-1}$ is obtained from $w$
by substituting $[bcb]$ for the $k$-th copy of $[aba]$ in $w$.
For instance, if $w_0 = [aba][cba]\eta[aba][aba]\eta$
then $N = 3$ and 
\[ I_2 = \{ [aba][cba]\eta[aba][aba]\eta,
[aba][cba]\eta[aba][bacb]\eta, [aba][cba]\eta[aba][bcb]\eta \}; \]
in this example,
$I_3 = \{w_0\} \subset I_2 \subset I_1 \subset I_0$;
$I_0$ has 27 elements.
We have
\begin{align*}
A_{2,+}([aba][cba]\eta[aba][bacb]\eta) &=
[aba][cba]\eta[aba][bacb]\eta, \\
A_{2,0}([aba][cba]\eta[aba][bacb]\eta) &=
[aba][cba]\eta[bacb][bacb]\eta, \\
A_{2,-}([aba][cba]\eta[aba][bacb]\eta) &=
[aba][cba]\eta[bcb][bacb]\eta. 
\end{align*}

For $w \in I_0$, let $\cM_{0}[w] = \cL_3[w]$
so that $\cM_{0}[w]$ is a contractible topological Hilbert manifold
(Theorem 2, \cite{Goulart-Saldanha1}).
For $0 < k \le N$, $w_1 \in I_k$,
recursively define
\[ \cM_{k}[w_1] = \cM_{k-1}[A_{k,-}(w_1)] \sqcup
\cM_{k-1}[A_{k,0}(w_1)] \sqcup \cM_{k-1}[A_{k,+}(w_1)].  \]
Every set $\cM_{k}[w_1]$ is a topological Hilbert manifold
(indeed, a submanifold of finite codimension of $\cL_3$)
and $\cM_{N}[w_0] = \cL_3^\bullet[w_0]$
(Equation \eqref{equation:Lbullet}).
We claim that every $\cM_k[w_1]$ is contractible:
the proof is by induction on $k$ and the base case $k = 0$
has already been taken care of.
For the induction step,
apply Lemma \ref{lemma:3contractible}
with $\cH = \cM_{k}[w_1]$, $\cH_0 = \cM_{k-1}[A_{k,0}(w_1)]$ and
$\cH_{\pm} = \cM_{k-1}[A_{k,\pm}(w_1)]$.
This completes the proof of the claim.
In particular, 
$\cM_{N}[w_0] = \cL_3^\bullet[w_0]$
is contractible, as desired.
\end{proof}

\begin{lemma}
\label{lemma:Mmu0mu1contractible}
For every $(\mu_0,\mu_1) \in \NN^2$,
the topological Hilbert manifold $\cM_{\mu_0,\mu_1} \subset \cL_{3}$
is contractible.
\end{lemma}

\begin{proof}
Let $M^\bullet_{\mu_0,\mu_1}$ be the finite set of words
$w_0 \in (M^\bullet)^\ast$ with $\mu(w_0) = (\mu_0,\mu_1)$:
\[ \cM_{\mu_0,\mu_1} =
\bigsqcup_{w_0 \in M^\bullet_{\mu_0,\mu_1}} \cL_3^\bullet[w_0]. \]
We want to show that $\cM_{\mu_0,\mu_1}$
can be obtained from these pieces $\cL_3^\bullet[w_0]$
(Equation \eqref{equation:Lbullet}, $w_0 \in M^\bullet_{\mu_0,\mu_1}$)
by applying Lemma \ref{lemma:3contractible} several times,
thus implying the contractibility of $\cM_{\mu_0,\mu_1}$.
The combinatorics of this process is best described
by considering an auxiliary finite dimensional construction.

If a convex polytope is cut by a finite family of hyperplanes,
the resulting pieces are also convex.
If we remove the hyperplanes one by one,
at each step we have some pieces which are left unchanged
and some triples which are glued back,
to produce a larger piece which is also convex.
For certain specific polytopes and hyperplanes,
the procedure described above is a guide
for glueing sets of curves.
At each step, the set of curves corresponding to a convex polytope
is a contractible topological Hilbert manifold.
We proceed to describe the polytopes, hyperplanes
and the correspondence.

For $d \in \NN$, let
\[ T_d = \{(t_1,\ldots,t_d) \in (0,1)^d \;|\; t_1 < \cdots < t_d \}
\subseteq \RR^d. \]
(For $d = 0$, $T_0 = \RR^0 = \{0\}$, $0 = ()$.)
The set $T_d$ is an open simplex. 
The set
$T_{\mu_0,\mu_1} = T_{\mu_0} \times T_{\mu_1} \subset \RR^{\mu_0+\mu_1}$
is also an open convex polytope. 
An element of $T_{\mu_0,\mu_1}$ can be written as
\begin{equation}
\label{equation:Tmu0mu1}
((t_{0,1},\ldots,t_{0,\mu_0}),(t_{1,1},\ldots,t_{1,\mu_1})),
\quad
t_{0,1} < \cdots < t_{0,\mu_0},
\quad
t_{1,1} < \cdots < t_{1,\mu_1}.
\end{equation}
For $1 \le i \le \mu_0$ and $1 \le j \le \mu_1$,
the hyperplane $H_{i,j} \subset \RR^{\mu_0+\mu_1}$
of equation $t_{0,i} = t_{1,j}$ crosses the set $T_{\mu_0,\mu_1}$.
The collection of all such hyperplanes
decomposes $T_{\mu_0,\mu_1}$ into convex strata
(the intersection with a hyperplane counts
as a stratum of higher codimension).
We shall see that such strata $T_{\mu_0,\mu_1}[w_0]$
are naturally labeled by $w_0 \in M^\bullet_{\mu_0,\mu_1}$,
with the codimension of $T_{\mu_0,\mu_1}[w_0] \subset T_{\mu_0,\mu_1}$
being given by $N_\eta$, the number of copies of $\eta$ in $w_0$.

Indeed, we describe a map from $T_{\mu_0,\mu_1}$ to $M^\bullet_{\mu_0,\mu_1}$.
Consider an element of $T_{\mu_0,\mu_1}$
as in Equation \eqref{equation:Tmu0mu1}.
Sort the numbers $t_{i,j}$ in increasing order
and read the first labels,
with $[aba]$ for $0$ and $[cba]$ for $1$:
if $t_{0,j_0} = t_{1,j_1}$ then the two letters fuse to give
a single letter $\eta$.
For instance, for $\mu_0 = \mu_1 = 3$,
the point $((1/9,4/9,5/9),(2/9,4/9,7/9))$
is mapped to the word $[aba][cba]\eta[aba][cba]$.
For $w_0 \in (M^\bullet)^\ast$, 
let $T_{\mu_0,\mu_1}[w_0] \subseteq T_{\mu_0,\mu_1}$
be the set of elements mapped to the word $w_0$.
Notice that for every $w_0 \in T_{\mu_0,\mu_1}$,
the subset $T_{\mu_0,\mu_1}[w]$ is nonempty and convex.
For instance, $T_{3,3}[[aba][cba]\eta[aba][cba]]$ is the convex subset
of $\RR^3 \times \RR^3$ defined by
\[ 0 < t_{0,1} < t_{1,1} < t_{0,2} = t_{1,2} < t_{0,3} < t_{1,3} < 1. \]
Thus,
\begin{equation}
\label{equation:Tmumu}
T_{\mu_0,\mu_1} =
\bigsqcup_{w_0 \in  M^\bullet_{\mu_0,\mu_1}}
T_{\mu_0,\mu_1}[w_0] 
\end{equation}
defines a stratification of $T_{\mu_0,\mu_1}$,
as claimed.

There is also a natural map
$\tau: \cM_{\mu_0,\mu_1} \to T_{\mu_0,\mu_1}$.
Indeed, for $\Gamma \in \cM_{\mu_0,\mu_1}$
consider its singular set $\sing(\Gamma) \subset (0,1)$,
the set of times $t \in (0,1)$
for which $\Gamma(t) \notin \Bru_\eta$.
For each $t \in \sing(\Gamma)$,
we have $\Gamma(t) \in \Bru_{\eta\sigma}$
for some $\sigma \in M$
($\sigma$ being a function of $t$).
Let $\{t_{0,1} < \cdots < t_{0,\mu_0}\} \subseteq \sing(\Gamma)$
be the set of times $t$ for which $\mu(\sigma) = (1,\ast)$.
Similarly,
let $\{t_{1,1} < \cdots < t_{1,\mu_1}\} \subseteq \sing(\Gamma)$
be the set of times $t$ for which $\mu(\sigma) = (\ast,1)$.
Notice that the intersection of these two sets
is the set of times $t$ for which $\sigma = \eta$.
Define
\[ \tau(\Gamma) =
((t_{0,1},\ldots,t_{0,\mu_0}),(t_{1,1},\ldots,t_{1,\mu_1})). \]
Notice that, for $w_0 \in M^\bullet_{\mu_0,\mu_1}$,
\[ \cL_3^\bullet[w_0] = \tau^{-1}[T_{\mu_0,\mu_1}[w_0]]. \]

The set of hyperplanes used to define the stratification
in Equation \eqref{equation:Tmumu} is labeled by
$I_{\mu_0,\mu_1} = \{(i,j) \in \ZZ^2, 1 \le i \le \mu_0, 1 \le j \le \mu_1 \}$.
A subset $I \subset I_{\mu_0,\mu_1}$ defines a 
coarser stratification of $T_{\mu_0,\mu_1}$
into convex strata.
Consider subsets $I_1 \subset I_0 \subseteq I_{\mu_0,\mu_1}$,
$I_0 \smallsetminus I_1 = \{(i_0,j_0)\}$,
and the stratifications corresponding to $I_0$ and $I_1$.
In order to pass from the $I_1$-stratification to the $I_0$-stratification,
some strata are not changed 
(if they are disjoint from the hyperplane $H_{i_0,j_0}$)
and some strata are cut into three parts
(if they are crossed by $H_{i_0,j_0}$).
Conversely, we pass
from the $I_0$-stratification to the $I_1$-stratification
by attaching certain triples of strata.

A subset $I \subseteq I_{\mu_0,\mu_1}$,
defines an equivalence relation in $M^\bullet_{\mu_0,\mu_1}$:
two words $w_0,w_1 \in M^\bullet_{\mu_0,\mu_1}$
are $I$-equivalent if no hyperplane $H_{i,j}$, $(i,j) \in I$,
separates the sets $T_{\mu_0,\mu_1}[w_0]$ and $T_{\mu_0,\mu_1}[w_1]$.
In particular, for $I = \varnothing$ all words are equivalent.
Let $E_I$ be a set of $I$-equivalent classes,
so that $W \in E_I$ is a set of words
and $T_{W,I} = \bigsqcup_{w_0 \in W} T_{\mu_0,\mu_1}[w_0]
\subseteq T_{\mu_0,\mu_1}$ is a convex polytope.
For $I = \varnothing$ we have $E_I = \{ M^\bullet_{\mu_0,\mu_1} \}$
and $T_{M^\bullet_{\mu_0,\mu_1},\varnothing} = T_{\mu_0,\mu_1}$.
The $I$-stratification is $T_{\mu_0,\mu_1} = \bigsqcup_{W \in E_I} T_{W,I}$.
Define
$\cM_{\mu_0,\mu_1}[W,I] = \tau^{-1}[T_{W,I}] \subseteq \cM_{\mu_0,\mu_1}$
so that, for any $I \subseteq I_{\mu_0,\mu_1}$
we have $\cM_{\mu_0,\mu_1} = \bigsqcup_{W \in E_I} \cM_{\mu_0,\mu_1}[W,I]$.

We claim that for every $I \subseteq I_{\mu_0,\mu_1}$
and every $W \in E_I$ the set $\cM_{\mu_0,\mu_1}[W,I]$
is contractible.
The proof is by induction on
$\delta(I) = |I_{\mu_0,\mu_1} \smallsetminus I|$.
The basis of the induction is the case $\delta(I) = 0$,
which has already been discussed.
The case $\delta(I) = \mu_0\mu_1 = |I_{\mu_0,\mu_1}|$
corresponds to $I = \varnothing$:
the proof of this case completes the proof of the lemma.

For the induction step,
consider subsets $I_1 \subset I_0 \subseteq I_{\mu_0,\mu_1}$,
$I_0 \smallsetminus I_1 = \{(i_0,j_0)\}$,
and the stratifications corresponding to $I_0$ and $I_1$.
By hypothesis, every set $\cM_{\mu_0,\mu_1}[W,I_0]$ ($W \in E_{I_0}$)
is contractible.
Consider a set
$\cM_{\mu_0,\mu_1}[W,I_1]$, $W \in E_{I_1}$
and its finite dimensional counterpart $T_{W,I_1}$.
If $T_{W,I_1}$ is not crossed by $H_{i_0,j_0}$
then $T_{W,I_1} = T_{W,I_0}$
and therefore 
$\cM_{\mu_0,\mu_1}[W,I_1] = \cM_{\mu_0,\mu_1}[W,I_0]$
is contractible.
If $T_{W,I_1}$ \textit{is} crossed by $H_{i_0,j_0}$
then $W = W^{(-)} \sqcup W^{(0)} \sqcup W^{(+)}$
with $W^{(-)}, W^{(0)}, W^{(+)} \in E_{I_0}$
and $T_{W^{(0)},I_0} \subset H_{i_0,j_0}$.
Apply Lemma \ref{lemma:3contractible}
with
$\cH_0 = \cM_{\mu_0,\mu_1}[W^{(0)},I_0]$,
$\cH_{+} = \cM_{\mu_0,\mu_1}[W^{(+)},I_0]$,
$\cH_{-} = \cM_{\mu_0,\mu_1}[W^{(-)},I_0]$
to deduce that $\cH = \cM_{\mu_0,\mu_1}[W,I_1]$
is contractible.
This completes the induction step,
the proof of the claim and the proof of the lemma.
\end{proof}

The following result is now easy.

\begin{lemma}
\label{lemma:coho}
The transversal bundle to $\cM_{\mu_0,\mu_1}$ is trivial
and therefore oriented.
Counting intersections with $\cM_{\mu_0,\mu_1}$
(with sign) defines an element 
$m_{\mu_0,\mu_1}$ of the cohomology groups
$H^{2\mu_0 + 2\mu_1}(\cL_3;\ZZ) = H^{2\mu_0 + 2\mu_1}(\cD_3;\ZZ)$.
\end{lemma}

\begin{proof}
We are here implicitly using Proposition~\ref{prop:submanifold}.
The triviality (and orientability) of the transversal bundle
follows from the contractibility of $\cM_{\mu_0,\mu_1}$.
This allows us to count intersections with signs.
The fact that $\cM_{\mu_0,\mu_1} \subset \cL_3$
is a closed subset (Proposition~\ref{prop:closed})
implies that the number of intersections is invariant
under homotopy or homology,
allowing for the definition of 
$m_{\mu_0,\mu_1} \in H^{2\mu_0 + 2\mu_1}(\cL_3;\ZZ)$.
The equality between cohomology groups of
$\cL_3$ and $\cD_3$ follows from the fact
that these two spaces are homotopy equivalent.
\end{proof}

At this point we have not ruled out the possibility
that $m_{\mu_0,\mu_1}$ could be trivial.
The fact that it is, indeed, nontrivial
will follow from the construction of the map
$h_{\mu_0,\mu_1}$ in the next section
(see in particular item 4 of Lemma~\ref{lemma:alphadois} below).

\bigbreak

\section{The maps $h_{\mu_0,\mu_1}$ and
Proposition~\ref{prop:M}}
\label{section:maps}

Following \cite{Goulart-Saldanha-cw} and
Remark~\ref{remark:abcaccba} above,
let $Y_2 = \{a,b,c,ba,ab,bc\} \subset S_4$.
Let $\Ideal_{Y_2} \subset \Word_3$ be the set of words
containing at least one letter of $Y_2$.
Let $\cL_3[\Ideal_{Y_2}] \subset \cL_3$
be the set of curves with itinerary in $\Ideal_{Y_2}$.
The set $\Ideal_{Y_2}$ is an \textit{ideal}
and therefore $\cL_3[\Ideal_{Y_2}] \subset \cL_3$
is an open subset.
Similarly, let $\cD_3[\Ideal_{Y_2}] \subset \cD_3$
be the union of cells with labels in $\Ideal_{Y_2}$:
$\cD_3[\Ideal_{Y_2}]$ is a subcomplex of $\cD_3$.

\begin{lemma}
\label{lemma:alpha0}
Consider $(\mu_0,\mu_1) \in \NN^2$, $\mu_0 + \mu_1 > 0$.
Let 
\[ w_\bullet = [aba]\cdots[aba][cba]\cdots[cba] \in M^\ast \subset \Word_3 \]
with $\mu_0$ copies of $[aba]$ and $\mu_1$ copies of $[cba]$
so that $\mu(w_\bullet) = (\mu_0,\mu_1)$.
There exists a continuous map
$\alpha^0_{\mu_0,\mu_1}: \DD^{2\mu_0+2\mu_1} \to \cD_3$
with the following properties:
\begin{enumerate}
\item{The restriction of $\alpha^0_{\mu_0,\mu_1}$
to $\Ss^{2\mu_0+2\mu_1-1}$ assumes values in $\cD_3[\Ideal_{Y_2}]$.}
\item{The map $\alpha^0_{\mu_0,\mu_1}$ covers
the cell $c_{w_\bullet}$ exactly once, in an injective manner.}
\item{The map $\alpha^0_{\mu_0,\mu_1}$
covers no other cell $c_w$ for $w \in M^\ast$.}
\end{enumerate}
\end{lemma}

\begin{remark}
\label{remark:alpha00}
The case $\mu_0 = \mu_1 = 0$ is left out due to its 
trivial and somewhat degenerate nature.
In order to extend the above result to this case
we should have $\DD^0$ be a set with a single point,
taken by $\alpha^0_{0,0}$ to $c_{()}$,
the isolated $0$-cell associated with the empty word.
In $\cL_3$, this cell corresponds
to the contractible connected component of convex curves,
which indeed have empty itinerary.
\end{remark}

\begin{proof}[Proof of Lemma \ref{lemma:alpha0}]
The map $\alpha^0_{1,0}$ is shown in Figure~\ref{fig:ababcb}
(left part).
More precisely, take $\DD^2$ injectively to $c_{[aba]}$,
so that its boundary is taken to the cells corresponding
to the eight cells (of dimensions $0$ and $1$)
corresponding to the eight words listed along the boundary.
Notice that all eight words belong to $\Ideal_{Y_2}$.
In other words, the boundary is taken to $\cD_3[\Ideal_{Y_2}]$.

The map $\alpha^0_{0,1}$ is shown in Figure~\ref{fig:cbaabc}.
The disk $\DD^2$ is taken injectively
to the disk $c_{[abc]} \cup c_{[ac]} \cup c_{[cba]} \subset \cD_3$.
Recall that $[cba] \in M$ but $[abc], [ac] \notin M$. 
As discussed in Remark~\ref{remark:abcaccba},
its boundary is taken to $\cD_3[\Ideal_{Y_2}]$.

The map $\alpha^0_{\mu_0,\mu_1}$ is constructed
by taking the cartesian product of $\mu_0$ copies of $\alpha^0_{1,0}$
followed by $\mu_1$ copies of $\alpha^0_{0,1}$.
The boundary is clearly taken to $\cD_3[\Ideal_{Y_2}]$.
\end{proof}

\begin{remark}
\label{remark:alpha0}
In Lemma \ref{lemma:alpha0} we choose to construct a map
assuming values in the CW complex $\cD_3$.
The duality (discussed in \cite{Goulart-Saldanha-cw})
between $\cD_3$ and $\cL_3$ tells us that
a similar construction can be performed in $\cL_3$.
For some parts of the construction
(particularly the last paragraph of the proof of Lemma \ref{lemma:alpha0})
working in $\cL_3$ intruduces unneeded technicalities
(thus our preference for $\cD_3$).
A brief discussion of the route not taken
does however help to clarify the situation.


The map $\tilde\alpha^0_{1,0}: \DD^2 \to \cL_3$ 
(dual to $\alpha^0_{1,0}$) is obtained by taking 
a curve in $\cM_{1,0}$ and perturbing it.
Indeed, start with a curve of itinerary $[aba]$
(to be $\tilde\alpha^0_{1,0}(0)$) and perturb it.
Figure~\ref{fig:ababcb}
shows us the resulting itineraries:
the boundary indicates the itineraries.
Notice that since we started with a curve of itinerary $[aba]$,
we are in the part where $\cM_{1,0}$ is a manifold.
A similar construction is performed in \cite{Saldanha3}
for the easier case of $\cL_2$ (curves in $\Ss^2$):
Figure~1 in \cite{Goulart-Saldanha1}
shows the resulting map (with itineraries).



Consider next the case $\mu_0 = 0$, $\mu_1 = 1$.
Start with a curve of itinerary $[cba]$ and perturb it.
The top of Figure \ref{fig:cbaabc} shows us
the resulting itineraries.
This includes the word $[ac]$, so is not satisfactory.
We can however use the bottom of the same figure as a guide:
take the perturbation further so as to
include in the image of the interior of $\DD^2$
a curve with itinerary $[abc]$.
Now the image of $\Ss^1$, the boundary of $\DD^2$,
has the itineraries in the boundary of the whole figure.
It would be interesting to draw the corresponding maps.
\end{remark}

\begin{remark}
\label{remark:alpha00b}
It would also be interesting to describe the maps
discussed in Remark \ref{remark:alpha00}
in the notation of \cite{Alves-Saldanha}.
A better understanding of Bruhat cells (and of itineraries)
in that notation is clearly necessary.
As mentioned in Remark \ref{remark:AS},
translating concepts into the notation of 
\cite{Alves-Saldanha} appears to be nontrivial.
\end{remark}



\begin{lemma}
\label{lemma:alpha1}
Consider $(\mu_0,\mu_1) \in \NN^2$, $\mu_0 + \mu_1 > 0$,
and the map $\alpha^0_{\mu_0,\mu_1}: \DD^{2\mu_0 + 2\mu_1} \to \cD_3$
constructed in Lemma \ref{lemma:alpha0}.
There exists a continuous map
$\alpha^1_{\mu_0,\mu_1}: [0,1] \times \partial\DD^{2\mu_0 + 2\mu_1} \to
\cD_3[\Ideal_{Y_2}]$ such that,
for all $s \in \partial\DD^{2\mu_0 + 2\mu_1}$,
$\alpha^1_{\mu_0,\mu_1}(0,s) = \alpha^0_{\mu_0,\mu_1}$
and
$\alpha^1_{\mu_0,\mu_1}(1,s) = aaaa \alpha^0_{\mu_0,\mu_1}$.
\end{lemma}

\begin{remark}
\label{remark:alpha1}
Also here we can find a simpler version of the construction
for the case $n = 2$ in \cite{Saldanha3}.
Indeed, Figure~6 in \cite{Saldanha3} is essentially a drawing
of the map $\alpha^1_1: [0,1] \times \Ss^1 \to \cL_2$.
\end{remark}

\begin{proof}[Proof of Lemma~\ref{lemma:alpha1}]
Indeed,
we know from Proposition 13.2 in \cite{Goulart-Saldanha-cw}
that any continuous map $f_0: \Ss^k \to \cD_3[\Ideal_{Y_2}]$
is homotopic in $\cD_3[\Ideal_{Y_2}]$ to
$f_1 = aaaaf_0: \Ss^k \to \cD_3[\Ideal_{Y_2}]$.
It suffices to apply this result to $f_0 = \alpha^0|_{\partial\DD}$.
\end{proof}

\begin{lemma}
\label{lemma:alphadois}
Consider $(\mu_0,\mu_1) \in \NN^2$, $\mu_0 + \mu_1 > 0$,
and the maps $\alpha^i_{\mu_0,\mu_1}$, $i \in \{0,1\}$,
constructed in Lemmas \ref{lemma:alpha0} and \ref{lemma:alpha1}.
Let $D = \DD^{2\mu_0 + 2\mu_1} \subset \RR^{2\mu_0 + 2\mu_1}$ and
$S = \partial([0,4] \times D) \subset \RR^{1+2\mu_0 + 2\mu_1}$,
a topological sphere of dimension $2\mu_0 + 2\mu_1$.
There exists a continuous map
$\alpha^{2}_{\mu_0,\mu_1}: S \to \cD_3$ with the following properties:
\begin{enumerate}
\item{The maps $\alpha^0$ and $\alpha^{2}$ coincide on
$\{0\} \times D$ (we identify $D$ and $\{0\} \times D$).}
\item{The maps  $\alpha^1$ and $\alpha^{2}$ coincide on
$[0,1] \times \partial D$.}
\item{The map $\alpha^{2}$ assumes values in $\cD_3[\Ideal_{Y_2}]$
on $S \smallsetminus (\{0\} \times D)$.}
\item{Let $m_{\mu_0,\mu_1} \in H^{2\mu_0 + 2\mu_1}(\cD_3;\ZZ)$
be as in Lemma~\ref{lemma:coho}.
Consider the pullback
$(\alpha^{2})^{\ast}(m_{\mu_0,\mu_1}) \in 
H^{2\mu_0 + 2\mu_1}(S;\ZZ) = \ZZ$:
this is a generator.
In particular, the map $\alpha^{2}: S \to \cD_3$
is not homotopic to a constant.}
\item{The map $aaaa\alpha^{2}: S \to \cD_3$
is homotopic to a constant.}
\end{enumerate}
\end{lemma}

\begin{remark}
\label{remark:loose}
The last two items in the statement of Lemma~\ref{lemma:alphadois}
are related to the concept of \textit{loose} maps,
as discussed in several previous papers,
starting with \cite{Saldanha3} and,
more recently and with more generality,
Section~10 of \cite{Goulart-Saldanha-cw}.
A map $\alpha: \Ss^k \to \cD_3$ is \textit{loose}
if and only if $\alpha$ is homotopic to $aaaa\alpha$
(Lemma~10.1 in \cite{Goulart-Saldanha-cw}).
The last two items thus imply that $\alpha^{2}$ is not loose.
The map $aaaa\alpha^{2}$ is loose.
\end{remark}

\begin{proof}[Proof of Lemma~\ref{lemma:alphadois}]
We construct the map $\alpha^{2}$ explicitly
and then prove the desired properties.
The first two items already give us $\alpha^{2}$
at the bottom $\{0\} \times D$
and at the first level of the wall,
$[0,1] \times \partial D$.

For $(h,s) \in [1,2] \times \partial D$
(so that $h \in [1,2]$ and $s \in \partial D$), 
define $\alpha^{2}((h,s)) = aaaa\alpha^0((2-h)s)$.
Notice that the map thus constructed is indeed continuous
at points $(1,s)$ 
(where both definitions give $aaaa\alpha^0(s)$).
Notice furthermore that
$\alpha^{2}((2,s)) = aaaa\alpha^0(0)$ (for all $s \in \partial D$).

Let $R$ be the orthogonal reflection on a hyperspace
in $\RR^{2\mu_0 + 2\mu_1}$
so that $R: D \to D$ is orientation reversing
and $R: \partial D \to \partial D$ has degree $-1$. 
For $(h,s) \in [2,3] \times \partial D$
define $\alpha^{2}((h,s)) = aaaa\alpha^0(R((h-2)s))
= \alpha^2(4-h,R(s))$.
For $(h,s) \in [3,4] \times \partial D$
define $\alpha^{2}((h,s)) = aaaa\alpha^1((h-3,R(s)))
= aaaa\alpha^2(h-3,R(s))$.
Finally, for $s \in D$ define
$\alpha^{2}((4,s)) = aaaaaaaa\alpha^0(R(s))$.
This completes the construction of the continuous map $\alpha^{2}$.
Its continuity is easily verified,
as are the first three items in the statement.
We are left with checking the last two items.

Let $i: \cD_3 \to \cL_3$ be the inclusion map
constructed in \cite{Goulart-Saldanha-cw} (especially in Theorem~4).
For item 4, notice that it follows from items (1) and (3)
that the map $i \circ \alpha^{2}: S \to \cL_3$
intersects $\cM_{\mu_0,\mu_1}$ precisely once at $(0,0) \in S$,
and that intersection is transversal.

For item (5),
we have to verify that $[aaaa\alpha^{2}]$ is zero
in the homotopy group
$G = \pi_{2\mu_0+2\mu_1}(\cD_3;\alpha^0(0))$
(which is abelian).
Since the connected components of $\cD_3$ (and of $\cL_3$)
are known to be simply connected
(from Theorem 2 in \cite{Goulart-Saldanha-cw}),
the base point is not an issue.

Let $S_1 = \partial([0,1] \times D)$.
Let $\beta: S_1 \to \cD_3$ coincide with $\alpha^{2}$
on the bottom and on the first level of the wall
and be defined on the top $\{1\} \times D$ by
$\beta((1,s)) = aaaa\alpha^0(s)$.
By construction,
$[\alpha^{2}] = [\beta] - [aaaa\beta] \in G$
and therefore
$[aaaa\alpha^{2}] = [aaaa\beta] - [aaaaaaaa\beta] \in G$.
The map $aaaa\alpha^{2}$ is clearly loose
(for instance, from Lemma~10.3 in \cite{Goulart-Saldanha-cw}):
from Lemma~10.1 in \cite{Goulart-Saldanha-cw}
we have $[aaaa\beta] = [aaaaaaaa\beta] \in G$.
This implies $[aaaa\alpha^{2}] = 0$, as desired.
\end{proof}

We are finally ready to prove Proposition~\ref{prop:M}.

\begin{proof}[Proof of Proposition~\ref{prop:M}]
The value of $z_1 \in \Quat_4$ such that
$\cM_{\mu_0,\mu_1} \subset \cL_3(1;z_1)$
is computed (using Theorem~2 in \cite{Goulart-Saldanha1}) as
\[ z_1 = \acute\eta(-1)^{\mu_0}(\hat a\hat c)^{\mu_1}\acute\eta
= (-1)^{(\mu_0+1)}(\hat a\hat c)^{(\mu_1+1)},
\quad -1 = \longhat(aba), \quad \hat a\hat c = \longhat(cba). \]
We know from Proposition~\ref{prop:closed}
that $\cM_{\mu_0,\mu_1} \subset \cL_3$ is a closed subset.
From Proposition~\ref{prop:submanifold},
$\cM_{\mu_0,\mu_1} \subset \cL_3$ is a topological submanifold.
From Lemma~\ref{lemma:Mmu0mu1contractible},
$\cM_{\mu_0,\mu_1}$ is contractible.

From Lemma~\ref{lemma:alphadois} above we have continuous maps
$\alpha^{2}_{\mu_0,\mu_1}: \Ss^{2\mu_0+2\mu_1} \to \cD_3$.
From Theorem~1 in \cite{Goulart-Saldanha-cw} 
we have a continuous map $i: \cD_3 \to \cL_3$:
take $h_{\mu_0,\mu_1} = i \circ \alpha^{2}_{\mu_0,\mu_1}:
\Ss^{2\mu_0+2\mu_1} \to \cL_3$.
We must verify that the map $h_{\mu_0,\mu_1}$
satisfies the desired properties.
For this, we need more information about the map $i$
which is also in \cite{Goulart-Saldanha-cw}
(starting with the statement of Theorem~4
and the subsequent definition of valid complexes).

The bottom $\alpha^0$ intersects $\cL_3[w_{\bullet}]$
precisely once and in a topologically transversal manner
(see the statement of Lemma~\ref{lemma:alpha0} above
and item (iv) in the definition of valid complexes).
Apart from this point, the bottom is taken to
$\cL_3[\Ideal^\ast(w_{\bullet})]$
(see item (v) in the same definition),
which is disjoint from $\cM$.
The other parts of the domain are taken to $\cL_3[\Ideal_{Y_2}]$,
which is also disjoint from $\cM$.
This proves items (2) and (3).
Finally, item (1) follows from items (4) and (5)
of Lemma~\ref{lemma:alphadois}.
\end{proof}

\bigbreak

\section{Proof of Theorems~\ref{theo:mainS3} and \ref{theo:L3}}
\label{section:theo}

We are ready to prove Theorem~\ref{theo:L3}.
Again, a comparison with the corresponding step
of the proof of the case $n=2$ (in \cite{Saldanha3})
should be illuminating.

Take $z_1 \in \{\pm 1, \pm \hat a\hat c\} = Z(\Quat_4)$.
Let $\tilde\cY(z_1) = \tilde\cY \cap \cL_3(1;z_1)$;
recall that $\tilde\cY$ is
the set of curves whose itinerary includes at least one letter in 
$\tilde Y = S_4 \smallsetminus (\{e\} \cup M)$
(see Equation~\ref{equation:MY}).
From Fact~\ref{fact:tildeY}, the inclusion
$\tilde\cY(z_1) \subset \Omega\Spin_4(1;z_1)$
is a weak homotopy equivalence.
For convenience, let $N_{z_1} \subset \NN^2$ be the set of pairs
$(\mu_0,\mu_1)$ with 
$z_1 = (-1)^{(\mu_0+1)}(\hat a\hat c)^{(\mu_1+1)}$.
We have
\[ \cL_3(1;z_1) = 
\tilde\cY(z_1)
\sqcup \bigsqcup_{(\mu_0,\mu_1) \in N_{z_1}}  \cM_{\mu_0,\mu_1}. \]
The idea is to start with $\tilde\cY(z_1)$
and add (or put back) the submanifolds $\cM_{\mu_0,\mu_1}$
while keeping track of the homotopy type.

\begin{remark}
\label{remark:dims}
For $z_1 = -\hat a\hat c$, we have $N_{z_1} = (2\NN) \times (2\NN)$.
The corresponding dimensions are $0$ (for $(\mu_0,\mu_1) = (0,0)$),
$4$ (for $(2,0)$), $4$ (for $(0,2)$), $8$, $8$, $8$ and so on,
with $j+1$ copies of $4j$.
For $z_1 = \hat a\hat c$, we have $N_{z_1} = (2\NN + 1) \times (2\NN)$.
The corresponding dimensions are
$2$ (for $(1,0)$), $6$ (for $(3,0)$), $6$ (for $(1,2)$) and so on,
with $j+1$ copies of $2+4j$.
For $z_1 = -1$, we start with $2$ (for $(0,1)$)
and in general also have $j+1$ copies of $2+4j$.
Finally, for $z_1 = 1$, we start with $4$ (for $(1,1)$)
and have $j+1$ copies of $4+4j$.
\end{remark}

We state a step in the proof as a lemma.

\begin{lemma}
\label{lemma:bouquetstep}
Consider $z_1 \in Z(\Quat_4)$ and
a finite subset $X_0 \subset N_{z_1}$.
Consider $(\mu_0,\mu_1) \in N_{z_1} \smallsetminus X_0$,
$(\mu_0,\mu_1) \ne (0,0)$,
and set $X_1 = X_0 \sqcup \{(\mu_0,\mu_1)\}$.
Set
\[ \cL_3(X_i) = 
\tilde\cY(z_1)
\sqcup \bigsqcup_{(\mu_0,\mu_1) \in X_i}  \cM_{\mu_0,\mu_1}
\subset \cL_3(1;z_1). \]
Then $\cL_3(X_1)$ is homotopy equivalent to
$\cL_3(X_0) \vee \Ss^{2\mu_0 + 2\mu_1}$.
\end{lemma}

\begin{proof}
Notice that both $\cL_3(X_0)$ and $\cL_3(X_1)$ are Hilbert manifolds
and that $\cL_3(X_1) = \cL_3(X_0) \sqcup \cM_{\mu_0,\mu_1}$.
Also, $\cM_{\mu_0,\mu_1} \subset \cL_3(X_1)$
is a closed subset and a contractible collared
topological submanifold of codimension $2\mu_0 + 2\mu_1$.
Draw a thin closed tubular neighborhood $V \subset \cL_3(X_1)$
around $\cM_{\mu_0,\mu_1}$; let $V_0$ be the interior of $V$.
By transversality,
we may assume that the intersection of the image of the map
$\alpha_{\mu_0,\mu_1}$ with $V$ is a closed disk $D$.
Let $W_0 = \cL_3(X_1) \smallsetminus V_0 \subset \cL_3(X_0)$
and $W_1 = W_0 \cup D \subset \cL_3(X_1)$.
Both inclusions $W_0 \subset \cL_3(X_0)$ and $W_1 \subset \cL_3(X_1)$
are deformation retracts and therefore homotopy equivalences.

The space $W_1$ is obtained from $W_0$ by glueing a disk $D$
of dimension $2\mu_0 + 2\mu_1$ along the boundary $\partial D$.
The map $\alpha_{\mu_0,\mu_1}$ tells us that $\partial D$
is homotopic to a point (in $W_0$).
Thus, $W_1$ is homotopy equivalent to $W_0 \vee \Ss^{2\mu_0 + 2\mu_1}$.
\end{proof}

\begin{remark}
\label{remark:veeS0}
The case $(\mu_0,\mu_1) = (0,0)$ was excluded to avoid confusion.
This rather degenerate case \textit{does} work, however,
provided we agree that $W_0 \vee \Ss^0 = W_0 \sqcup \{\bullet\}$.
The contractible connected component corresponding to $\{\bullet\}$
is of course the set of convex curves.

Also, if $z_1 = -\hat a\hat c$ and $(0,0) \in X_0$,
the set $\cL_3(X_0)$ has two connected components,
one of them contractible.
We should agree that the base point at which the sphere is glued in 
$\cL_3(X_0) \vee \Ss^{2\mu_0 + 2\mu_1}$
belongs to the non-contractible connected component.
\end{remark}

\begin{proof}[Proof of Theorem \ref{theo:L3}]
Consider $z_1 \in Z(\Quat_4)$.
In order to determine the homotopy type of $\cL_3(z_1)$,
start with $\tilde\cY(z_1)$ and put back the subsets 
$\cM_{\mu_0,\mu_1}$ for $(\mu_0,\mu_1) \in N_{z_1} \subset \NN^2$.
The dimensions are described in Remark~\ref{remark:dims};
the effect of each step is described in Lemma~\ref{lemma:bouquetstep}.
\end{proof}

\begin{proof}[Proof of Theorem \ref{theo:mainS3}]
We have
\[ \cL_3(I;I) = \cL_3(1;1) \sqcup \cL_3(1;-1), \qquad
\cL_3(I;-I) = \cL_3(1;\hat a\hat c) \sqcup \cL_3(1;-\hat a\hat c). \]
The homotopy type of the connected components of
$\cL_3(I;\pm I)$ is therefore given in Theorem~\ref{theo:L3}.

Given a $1$-periodic curve $\Gamma: \RR \to \SO_{n+1}$,
$\Gamma \in \cL_3(\per)$,
consider $Q = \Gamma(0)$ and $\Gamma_0 = Q^{-1} \Gamma$.
The map $\Phi: \cL_3(\per) \to \cL_3(I;I) \times \SO_4$,
$\Phi(\Gamma) = (\Gamma_0,Q)$,
is a homotopy equivalence.
It is not surjective, however,
and therefore not a homeomorphism.
Let us first comment on this difficulty:
we shall than see that it can be overcome.

Recall from Remark \ref{remark:sufficientlysmooth} that
we were vague concerning the differentiability class
of the curves $\gamma$ or $\Gamma$:
assume that we have $\Gamma \in H^r$, $r \ge 5$.
The space of curves $\cL_3(I;I)$ is the set of curves
$\Gamma: [0,1] \to \SO_{n+1}$ in $H^r$ with
$\Gamma(0) = \Gamma(1) = I$.
Notice that we do not require, for instance,
$\Gamma'(0) = \Gamma'(1)$.
Such a curve $\Gamma$ can be extended in a $1$-periodic manner
to $\Gamma_1: \RR \to \SO_{n+1}$.
The curve $\Gamma_1$ need not, however, be in $H^r$,
since it may not be differentiable at $t = 0$.

In order to define $\Psi: \cL_3(I;I) \times \SO_4 \to \cL_3(\per)$
at $(\Gamma,Q)$, we need not just extend $\Gamma$ in a periodic manner.
We also need to change the periodic curve $\Gamma_1$ slightly,
so as to make it sufficiently smooth at $t = 0$,
thus defining another curve $\Gamma_2$,
also $1$-periodic and with $\Gamma_2(0) = I$.
The smoothening procedure has been discussed at length
in Lemma~2.5 in \cite{Goulart-Saldanha1}.
We then define $\Psi(\Gamma,Q) = Q\Gamma_2$.
The compositions $\Phi \circ \Psi$ and $\Psi \circ \Phi$
are homotopic to the identity, completing the proof.
\end{proof}

\bibliography{gs}
\bibliographystyle{plain}

\footnotesize

\noindent
Emília Alves \\
Departamento de Matem\'atica Aplicada,
Instituto de Matem\'atica e Estat{\'\i}stica, UFF, \\
Rua Professor Waldemar de Freitas Reis s/n,
Niter\'oi, RJ 24210-201, Brazil \\
\url{emiliaalves@id.uff.br}


\noindent
Victor Goulart \\
Departamento de Matem\'atica, UFES, \\
Av. Fernando Ferrari 514; Campus de Goiabeiras, Vit\'oria, ES 29075-910, Brazil. \\
\url{jose.g.nascimento@ufes.br}

\smallskip

\noindent
 Nicolau C. Saldanha \\
Departamento de Matem\'atica, PUC-Rio, \\
R. Marqu\^es de S. Vicente 255, Rio de Janeiro, RJ 22451-900, Brazil.  \\
\url{saldanha@puc-rio.br}

\end{document}